\pdfoutput=1
\documentclass[12pt,reqno]{amsart}
\usepackage{amsmath}
\usepackage{amsfonts}
\usepackage{amscd}
\usepackage{amssymb}
\usepackage{amsthm}
\usepackage{mathdots}
\usepackage{mathtools}
\usepackage{bm}
\usepackage[all]{xy}
\usepackage[linktocpage=true]{hyperref}
\usepackage{mathrsfs}

\usepackage{microtype}

\usepackage{tikz,tikz-cd, color}
\usetikzlibrary{shapes,arrows,positioning}

\usepackage{wasysym, stackengine, makebox, graphicx}

\usepackage{adjustbox}
\usetikzlibrary{matrix}
\usetikzlibrary{decorations.pathmorphing}

\tikzset{
  symbol/.style={
    draw=none,
    every to/.append style={
      edge node={node [sloped, allow upside down, auto=false]{$#1$}}}
  }
}

\usepackage{stmaryrd}
\usepackage{enumerate}
\usepackage{xcolor}
\usepackage{booktabs}
\usepackage{float}
\mathchardef\mhyphen="2D % Define a "math hyphen"

\newtheorem{thm}{Theorem}
\newtheorem*{thm*}{Theorem}

\newtheorem{proposition}[thm]{Proposition}
\newtheorem{corollary}[thm]{Corollary}

\newtheorem{thm&defn}[thm]{Theorem \& Definition}

\newtheorem{exmp}[thm]{Example}

\newtheorem{theorem}[thm]{Theorem}
\newtheorem{lemma}[thm]{Lemma}
\newtheorem{lm}[thm]{Lemma}

\newtheorem{fact}[thm]{Fact}

\theoremstyle{definition}

\newtheorem{definition}[thm]{Definition}
\newtheorem*{notation*}{Notation}
\newtheorem{example}[thm]{Example}
\newtheorem{remark}[thm]{Remark}
\newtheorem{rmk}[thm]{Remark}

\theoremstyle{remark}

\newtheorem*{rem*}{Remark}

\numberwithin{equation}{section}
\numberwithin{thm}{section}
\usepackage{fullpage}

\newcommand{\N}{\mathbb{N}}
\newcommand{\Z}{\mathbb{Z}}
\newcommand{\Q}{\mathbb{Q}}
\newcommand{\R}{\mathbb{R}}
\newcommand{\C}{\mathbb{C}}

\newcommand{\lb}{\text{\tiny $\bullet$}}

\newcommand{\eZ}{\mathbb{Z}_{\geq 0}}

\renewcommand{\P}{\mathbb{P}}
\renewcommand{\O}{\mathcal{O}}

\newcommand{\red}{\mathrm{red}}

\newcommand{\bZ}{\mathbb Z}
\newcommand{\bN}{\mathbb N}
\newcommand{\bQ}{\mathbb Q}
\newcommand{\bR}{\mathbb R}
\newcommand{\bC}{\mathbb C}

\newcommand{\bfV}{\mathbf{V}}

\newcommand{\cR}{\mathcal{R}}

\newcommand{\fp}{\mathfrak{p}}

\newcommand{\sA}{\mathscr{A}}
\newcommand{\sB}{\mathscr{B}}

\newcommand{\sI}{\mathscr{I}}

\newcommand{\sfS}{\mathsf{S}}
\newcommand{\sfSp}{\mathsf{S}_{\mathrm{split}}}

\newcommand{\Graphc}{\overline{\Gamma}_f}

\newcommand{\tSigma}{\widetilde{\Sigma}}

\DeclareMathOperator{\inte}{int}
\DeclareMathOperator{\Tor}{Tor}

\DeclareMathOperator{\Hom}{Hom}

\DeclareMathOperator{\cNE}{\overline{NE}}

\DeclareMathOperator{\Exc}{Exc}
\DeclareMathOperator{\Spec}{Spec}

\DeclareMathOperator{\Cone}{Cone}
\DeclareMathOperator{\rk}{rank}
\DeclareMathOperator{\codim}{codim}

\DeclarePairedDelimiter{\gen}{\langle}{\rangle}

\newcommand{\quotient}[2]{{\raisebox{.2em}{$#1$}\!\!\left/\!\!\raisebox{-.2em}{$#2$}\right.}}

\newcommand{\inclusion}{\hookrightarrow}

\ProvideDocumentCommand{\xrightleftarrows}{ O{}m }{%
            \mathrel{%
            \vcenter{\hbox{%
            \begin{tikzpicture}
              \node[minimum width=0.2cm,minimum height=1ex,anchor=south,align=center] (a){\vphantom{hg}${\scriptstyle #2}$\\ \vphantom{hg}${\scriptstyle #1}$};
              \draw[->] ([yshift=0.35ex]a.west) -- ([yshift=0.35ex]a.east);
              \draw[<-] ([yshift=-0.35ex]a.west) -- ([yshift=-0.35ex]a.east);
            \end{tikzpicture}}}%
            }%
            }

\newcommand{\ignore}[1]{}

\usepackage{accents}

\title{Fiber products under toric flops and flips}

\author{Tsung-Chen Chen}
\address[]{Department of Mathematics, National Taiwan University, Taipei 10617, Taiwan}
\email{d13221003@ntu.edu.tw}

\author{Hui-Wen Lin}
\address[]{Department of Mathematics and Taida Institute for Mathematical Sciences (TIMS), National Taiwan University, Taipei 10617, Taiwan}
\email{linhw@math.ntu.edu.tw}

\author{Sz-Sheng Wang}
\address[]{Department of Applied Mathematics, National Yang Ming Chiao Tung University, Hsinchu 30010, Taiwan}
\email{sswangtw@math.nctu.edu.tw}

\begin{document}

\begin{abstract} 
Let $\Sigma$ and $\Sigma'$ be two refinements of a fan $\Sigma_0$ and $f \colon X_{\Sigma} \dashrightarrow X_{\Sigma'}$ be the birational map induced by $X_{\Sigma} \rightarrow X_{\Sigma_0} \leftarrow X_{\Sigma'}$. We show that the graph closure $\Graphc$ is a \emph{not necessarily normal} toric variety and we give a combinatorial criterion for its normality. 

In contrast to it, for $f$ being a toric flop/flip, we show that the scheme-theoretic fiber product $X \coloneqq X_{\Sigma}\times_{X_{\Sigma_0}}X_{\Sigma'}$ is in general \emph{not toric}, though it is still \emph{irreducible} and $X_{\rm red} = \overline{\Gamma}_f$. 

A complete numerical criterion to ensure $X = X_{\rm red}$ is given for 3-folds, which is fulfilled when $X_\Sigma$ has at most terminal singularities. In this case, we further conclude that $X$ is normal. 

The result turns out to be a key step to study the correspondence between Chow motives under toric flops and flips.
\end{abstract}

\maketitle

%%%%%%%%%%%%%%%%%%%%%%%%%%%%%%%%%%%%%%%%
%%%%%%%%%%%%%%%%%%%%%%%%%%%%%%%%%%%%%%%%
\section{Introduction}

For two schemes $Y$ and $Y'$ over a scheme $S$, the fiber product $Y\times_SY'$ exists and is unique up to isomorphisms.
It is natural to ask ``Is the scheme-theoretic fiber product of two toric varieties still toric?" Unfortunately, the answer is ``NO". For example, let $\sigma = \Cone( e_1, e_2)$ and $\sigma' = \Cone (e_1, e_1+e_2)$ which associate two affine toric varieties $U_\sigma = \operatorname{Spec} \mathbb{C}[x^{e_1^\vee}, x^{e_2^\vee}]$ and $U_{\sigma'} = \operatorname{Spec} \mathbb{C}[x^{e_1^\vee}, x^{e_2^\vee}, x^{e_1^\vee-e_2^\vee}]$. It is easy to see that
\[U_{\sigma'}\times_{U_{\sigma}}U_{\sigma'}
%=\Spec \C[u, v, w]}/\langle u-vw\rangle\otimes_{\C[u, v]}\C[u, v, w']/\langle u-vw'\rangle
=\Spec \C[u, v, w, w']/\langle u-vw, v(w-w')\rangle\]
and $\langle u-vw, v(w-w')\rangle$ is not a prime ideal, so $U_{\sigma'}\times_{U_{\sigma}}U_{\sigma'}$ is not a toric variety. 

In birational geometry, the minimal model program plays an important role, in which the key ingredients consist of flops and flips, so it is a good choice to study toric fiber products under flops and flips. Actually the study of toric fiber products is related to the study of graph closures.

In this paper, we study the fundamental problem : {\it what is the difference between the graph closure and the fiber product under toric flops and flips.} \footnote{For our terminology of toric flops and flips, see Remark \ref{rmk_toricflip}.} Notice that we use the definition of toric varieties by constructing them from a finite subset $\sA$ of the lattice $M$ of characters of the torus (see Section \ref{afftoric_subsec}) and thus the toric varieties may be {\it not normal}. But if we regard $\sA$ as a generating set in $M$ and set $\sigma = \Cone (\sA)^\vee \subseteq N_\bR$ where $N$ is the dual lattice of $M$,  by \cite{CLS}, 
$\sigma$ is a strongly convex rational polyhedral cone and 
then $\Spec (\bC [\sfS_\sigma])$ is the normalization of our original toric variety constructed from $\sA$.

Since a subscheme of a toric variety may not be a toric variety, the first thing we have to do is to study the toric structure of graph closures and fiber products under toric flops and flips. For the toric structure of graph closures, we can consider a more general birational map from two refinements of a fan. Indeed, given a fan $\Sigma_0$ in $N_\bR$, a fan $\Sigma$ refining $\Sigma_0$ will yield a toric morphism $X_{\Sigma} \to X_{\Sigma_0}$.
In Theorem\autoref{gfbir_prop}, let $\Sigma$ and $\Sigma'$ be two refinements of a fan $\Sigma_0$ and let $f \colon X_{\Sigma} \dashrightarrow X_{\Sigma'}$ be the birational map induced by $X_{\Sigma} \rightarrow X_{\Sigma_0} \leftarrow X_{\Sigma'}$. Then 
{\it the graph closure $\Graphc$ is a toric variety}. Moreover, if we define the coarsest common refinement of $\Sigma$ and $\Sigma'$ over $\Sigma_0$ by
\begin{equation*}%\label{fp_fan}
    \tSigma \coloneqq \{ \sigma \cap \sigma' \mid \sigma\in\Sigma,\sigma'\in\Sigma' \mbox{ such that } \sigma, \sigma' \subseteq \sigma_0 \mbox{ for some } \sigma_0 \in \Sigma_0 \}
\end{equation*}
in $N_\bR$, then {\it the normalization of $\Graphc$ is the toric variety $X_{\tSigma}$}.

% The fiber product of toric morphisms is discussed in detail in Appendix \ref{app_fiber_product}. We define the fiber product in the category of fans and prove Theorem\autoref{genfp_thm}, which is a generalization of Theorem\autoref{gfbir_prop}.

%For general toric morphism, the fiber product of torus is isomorphic to $T\times G$, where $T$ is a torus and $G$ is a finite group. We can define the fiber product in the category of fans as in Appendix \ref{app_fiber_product}, and we prove that the corresponding toric variety is the normalization of $\overline{T\times e_G}$, where $e_G$ is the identity element in $G$.

%In particular, we can give a combinatorial criterion for the normality of graph closures. 
The following is a combinatorial criterion for the normality of graph closures. 

\begin{theorem}{\rm(= Theorem\autoref{gfcri_thm})}
The following statements are equivalent:
\begin{enumerate}[(i)]
    \item%\label{gfcri_thm1}
    $\overline{\Gamma}_f$ is normal.

    \item%\label{gfcri_thm2}
    $\overline{\Gamma}_f$ is the toric variety  associated with the fan $\tSigma$.
    
    \item%\label{gfcri_thm3} 
    For all (maximal cone) $\sigma$ and $\sigma'$ contained in some cone of $\Sigma_0$,
    \begin{equation*}%\label{cone+cone=cone}
        \sfS_{\sigma} + \sfS_{\sigma'} = \sfS_{\sigma \cap \sigma'}.
    \end{equation*}
\end{enumerate}
\end{theorem}

Using the combinatorial criterion, in Corollary\autoref{gftoric_cor}, we prove that if the toric variety $X_{\Sigma}$ is smooth, then $\Graphc$ is normal  for a toric flop or flip $f$. Moreover, if $n=3$ and $f \colon X_\Sigma \dashrightarrow X_{\Sigma'}$ is a toric flip or flop, then to achieve the normality of $\Graphc$, we only need the smoothness of one affine piece of $X_{\Sigma}$. 

The geometric picture of a toric flip is outlined as follows. For smooth toric 3-folds, Danilov in late 70's proved that one can move the fan $\Sigma$ to $\Sigma'$ by a sequence of elementary flops. In each step, it corresponds to the move 
% \small
% $$
% \xymatrix{&u_1\ar@{-}[ld]\ar@{-}[dd]\ar@{-}[rd]\\
% u_3\ar@{-}[rd]& &u_4\ar@{-}[ld]\\ &u_2}
% \mathop{\Longrightarrow}^{\hbox{flop}}
% \xymatrix{&u_1\ar@{-}[ld]\ar@{-}[rd]\\
% u_3\ar@{-}[rd]\ar@{-}[rr]& &u_4\ar@{-}[ld]\\ &u_2}
% $$ \normalsize
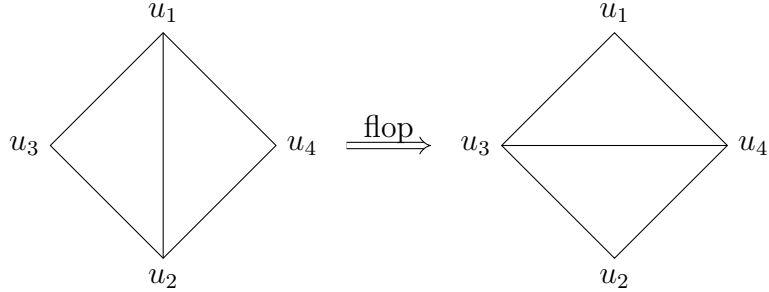
\[\begin{tikzpicture}[scale=0.75]
    % Left quadrilateral
    \coordinate (u1) at (-1,2);
    \coordinate (u2) at (-1,-2);
    \coordinate (u3) at (-3,0);
    \coordinate (u4) at (1,0);

    \draw (u1) -- (u3) -- (u2) -- (u4) -- cycle;
    \draw (u1) -- (u2);

    \node at (u1) [above] {$u_1$};
    \node at (u2) [below] {$u_2$};
    \node at (u3) [left] {$u_3$};
    \node at (u4) [right] {$u_4$};

    % Right quadrilateral (flopped)
    \coordinate (u1r) at (7,2);
    \coordinate (u2r) at (7,-2);
    \coordinate (u3r) at (5,0);
    \coordinate (u4r) at (9,0);

    \draw (u1r) -- (u3r) -- (u2r) -- (u4r) -- cycle;
    \draw (u3r) -- (u4r);

    \node at (u1r) [above] {$u_1$};
    \node at (u2r) [below] {$u_2$};
    \node at (u3r) [left] {$u_3$};
    \node at (u4r) [right] {$u_4$};

    % Arrow between diagrams
    \draw[double distance=2pt, -implies] (2.25,0) -- (3.75,0);
    \node at (3,0.3) {flop};
\end{tikzpicture}\]
where the primitive vectors $u_i$ lie in a plane and $u_1 + u_2 =
u_3 + u_4$. Geometrically this is the blowing-up of a $(-1,-1)$
rational curve in a 3-fold then followed by a blowing-down of the
exceptional divisor $\mathbb{P}^1\times \mathbb{P}^1$ in another
direction, that is, it is an easy ordinary flop. Reid in early 80's generalized the above elementary move to higher
dimensions: let $u_1,\ldots,u_{n+1}$ be primitive vectors in
$N=\mathbb{Z}^n$ such that $\sigma_n$ and $\sigma_{n+1}$ be two
top dimensional cones intersect along the face cone
$\tau= \Cone(u_1,\ldots,u_{n-1})$, where
$\sigma_j\coloneqq\Cone(u_1,\ldots,\hat{u_j},\ldots,u_{n+1})$. Let
the linear relation between $u_i$'s be
$$
b_1 u_1 + \cdots + b_n u_n + u_{n+1} = 0,
$$
which is called a wall relation. Here we set $b_{n+1}=1$ and we must have $b_n > 0$ since $u_n$ and
$u_{n+1}$ lie in opposite sides of $\tau$. Reordering
$u_1,\ldots,u_{n-1}$ we may assume that $b_i < 0$ for $1\le i\le
\alpha$, $b_i=0$ for $\alpha + 1\le i\le \beta$ and $b_i > 0$ for
$\beta+1\le i\le n+1$. Notice that $0\le \alpha\le \beta\le
n-1$.

The case one can perform ``elementary move'' is when $\alpha\ge
2$. Then there are two different decompositions of
$\sigma_0\coloneqq\Cone(u_1,\ldots,u_{n+1})$:
$$
\sigma_0 = \bigcup\nolimits_{\beta + 1\le j\le n+1}\sigma_j =
\bigcup\nolimits_{1\le j\le \alpha}\sigma_j.
$$
Our original two cones are in the first decomposition and the second decomposition will give us the local construction of a {\it toric flip}, whose global definition is given in Section \ref{ToricFlips}. When $u_1,\ldots,u_{n+1}$ lie in an affine hyperplane of $N_\Q$, it leads to a {\it toric flop}.

Next, for toric flops and flips, we study the toric structure of their fiber products. In Proposition\autoref{fpcri_prop}, 
we find that the fiber product $X \coloneqq X_{\Sigma}\times_{X_{\Sigma_0}}X_{\Sigma'}$ is a normal toric variety if and only if 
the scheme $X$ is (1) irreducible, (2) reduced and (3) the graph closure $\overline{\Gamma}_f$ is a normal toric variety. The condition (3) is studied in Section \ref{gf_sec} mentioned as above. For the condition (1), we show that $X$ is always irreducible under a toric flop or flip. Moreover, the reduced scheme $X_{\text{red}}$ with respect to $X$ is the toric variety $\Graphc$. In particular, if $X$ is a normal toric variety, then $X=\Graphc=X_{\tSigma}$.  
\begin{theorem}{\rm (=Theorem\autoref{fpred_thm})}
Let $f \colon X_{\Sigma} \dashrightarrow X_{\Sigma'}$ be a toric flip via toric morphisms $X_{\Sigma} \rightarrow X_{\Sigma_0} \leftarrow X_{\Sigma'}$ \eqref{tflip_eqn} and $X = X_{\Sigma} \times_{X_{\Sigma_0}} X_{\Sigma'}$ be the fiber product. Then we have:
\begin{enumerate}[(i)]
    \item%\label{fpred_thm1} 
    The reduced scheme $X_{\red}$ associated to $X$ is a (not necessarily normal) toric variety.

    \item%\label{fpred_thm2} 
    The normalization of $X_{\red}$ is the toric variety $X_{\tSigma}$. 
\end{enumerate}
\end{theorem}
The condition (2) of Proposition\autoref{fpcri_prop} is the biggest cause of uncertainty. So far, some useful criteria are given only for 3-dimensional case. Since the property of being reduced is local, we usually assume $X_{\Sigma_0}$ is an affine toric variety $U_{\sigma_0}$ defined by $\sigma_0$ and $U_{ji}\coloneqq U_{\sigma_j}\times_{U_{\sigma_0}}U_{\sigma_i}$. 
Also, we can assume $b_i\in\Z$, $\gcd(b_1, b_2, b_3, b_4) = 1$, and $b_4=1$ by the assumption $U_{\sigma_4}$ is smooth.
We provide a numerical criterion for the reduced property on the affine piece $U_{31}$. %one affine piece.
\begin{theorem}{\rm (=Lemma\autoref{n=3_reduced_equiv})}
Let $\{a\}_b$ denote the remainder of $a$ divided by $b$. If $g=\operatorname{gcd}(b_1,b_2)>0$ and $b_i=-gb_i'$ for $i=1$, $2$, then the following statements are equivalent:
\begin{enumerate}[(i)]
    \item $U_{31}$ is reduced.

    \item For all $0\leq\lambda\leq b_1'b_2'$, there exists a non-negative integer $y\leq\lambda/b_1'$ such that
    \begin{equation}
        \{g\lambda\}_{b_3}\geq g\cdot\{\lambda-b_1'y\}_{b_2'}.
    \end{equation}
    %where $\{n\}_m$ denotes the remainder of $n$ divided by $m$.
\end{enumerate}
\end{theorem}

Finally, we show that the reduced property of the special affine piece implies the reduced property of the whole $X$. 
\begin{theorem}{\rm (=Theorem\autoref{r=1 reduced})}
$X$ is reduced if and only if $U_{31}$ is reduced.
\end{theorem}

As an application, we get the following expected result for 3-dimensional case. 
\begin{theorem}{\rm (=Theorem\autoref{term3flip_thm})}
If $X_{\Sigma}$ is  a $3$-dimensional simplicial toric variety with at worst terminal singularities, then the fiber product $X = X_{\Sigma}\times_{X_{\Sigma_0}}X_{\Sigma'}$ is the toric variety $X_{\tSigma}$.
\end{theorem}
A generalized version of Lemma\autoref{n=3_reduced_equiv} can also help us to get a criterion for the smooth higher dimensional case. 
\begin{theorem}{\rm (=Theorem\autoref{smooth is redcued equivalent})}
Assume that $X_{\Sigma}$ is smooth of dimension $n$ and $X_{\Sigma_0}$ is affine with the wall relation \eqref{wallrel_eqn2}. Then the fiber product $X=X_{\Sigma}\times_{X_{\Sigma_0}}X_{\Sigma'}$ is the toric variety $X_{\tSigma}$ %is reduced 
if and only if 
\begin{equation*}%\label{divided}
b_{i}\mid b_{j} \mbox{ or }\ b_{j}\mid b_i
\end{equation*}
for any $i$, $j\in J_-$.
\end{theorem}

This work is also motivated by the conjecture given by C.-L.Wang in 2001 (ref. \cite{Wang01}), which said that for $K$-equivalent manifolds under birational map $f\colon X\dashrightarrow X'$, there is a {\it naturally attached} correspondence $T\in A^{\dim X}(X\times X')$ of the form $T = \bar\Gamma_f + \sum_i T_i$ with $\bar\Gamma_f\subseteq X\times X'$ the cycle of graph closure of $f$ and with $T_i$'s being certain degenerate correspondences (i.e.\ $T_i$ has positive dimensional fibers when projecting to $X$ or $X'$) such that $T$ is an isomorphism of Chow motives. In \cite{LLW10}, the authors showed that for an ordinary $\P^r$ flop $f\colon X \dashrightarrow X'$, the graph closure $[\bar\Gamma_f] \in A^*(X \times X')$ identifies the Chow motives $\hat{X}$ of $X$ and $\hat{X'}$ of $X'$. More generally, for $f$ an ordinary $(r, r')$ flip with $r \le r'$, the graph closure $[\bar\Gamma_f] \in A^*(X \times X')$ identifies the Chow motive $\hat X$ of $X$ as a sub-motive of $\hat X'$ which preserves also the Poincar\'e pairing on cohomology groups. 

In toric case, we have the following observation which is well-known for experts. (For example, cf. \cite{CLS} or \cite{Kawa16}.)
\begin{remark}{\rm (= Proposition\autoref{smooth=ordinary} + Theorem\autoref{K-equiv})}
\begin{enumerate}[(i)]
\item Any two $K$-equivalent simplicial %$\mathbb{Q}$-factorial 
terminal toric varieties can be connected to each other by a sequence of
toric flops.
\item Any smooth toric flop is an ordinary flop.
\end{enumerate}
\end{remark}

We would like to give a simple proof in Appendix \ref{app_sec} by supplementing Reid's theory on toric minimal model program. Now, together with the result in \cite{LLW10}, if $K$-equivalent toric manifolds are connected by smooth flops, then they admit canonically isomorphic integral cohomology groups via the graph closure. To study the conjecture for toric manifolds in \cite{Wang01}, it is important to know for a general toric flop whether the graph closure still provides a canonical isomorphism between the integral cohomology groups, or we need extra degenerate correspondence in their fiber product.

{\bf As an application}, under a 3-dimensional terminal toric flop $f$, its fiber product is equal to the graph closure and thus is expected to give the equivalence of their Chow motives. 
%%%%%%%%%%%%%%%%%%%%
%%%%%%%%%%%%%%%%%%%%
\medskip

\emph{Acknowledgments.} 
This paper is a continuation of the undergraduate thesis of T.-C.~Chen at National Taiwan University under the supervision of H.-W.~Lin. We wish to thank J.-H.~Chong and S.-Y.~Lee for useful discussions related to this paper. In particular, we appreciate that C.-L.~Wang provided some geometric point of view for this work. H.-W.~Lin and S.-S.~Wang are supported by the National Science and Technology Council (NSTC), and S.-S. Wang in part by Grant 114-2115-M-A49-007-MY3.. We are grateful to Taida Institute of Mathematical Sciences (TIMS) for its constant support which makes this collaboration possible.

%This paper is a continuation of my undergraduate thesis under the guidance of H.-W.~Lin, for which I am deeply grateful. We wish to thank J.-H.Chong and S.-Y.Lee for useful discussions related to this paper. In particular, we appreciate that C.-L.Wang provided some geometric point of view for this work. H.-W.~Lin and S.-S.Wang are supported by the National Science and Technology Council (NSTC). We are grateful to Taida Institute of Mathematical Sciences (TIMS) for its constant support which makes this collaboration possible.

%%%%%%%%%%%%%%%%%%%%
%%%%%%%%%%%%%%%%%%%%
\section{Preliminaries}\label{pre_sec}

%We begin by recalling In this section, we recall the basic notions of toric varieties and fix our notation. For the proofs, see ...[Oda78, Dan78, Ful93]

%A variety is a separated, integral scheme of finite type over $\bC$.

We begin by recalling the basic notions of toric varieties and fixing our notation. For further details, see \cite{CLS} and \cite{Fulton93}.

%\cite[Definition 3.1.1 & Chap.3 Appendix]{CLS}
A toric variety is a (not necessarily normal) variety $X$ containing an algebraic torus as a Zariski open subset, together with an algebraic action of the torus on $X$ that extends the natural action of the torus on itself. Let $M \simeq \bZ^n$ be the lattice of characters of the torus, with dual lattice $N \coloneqq \Hom_\bZ (M, \bZ)$. Then the torus is canonically isomorphic to $N \otimes_\bZ \bC^\ast$, denoted by $T_N$. Throughout this paper, we fix the lattices $M$ and $N$ of rank $n$, and their $\bR$-linear extensions $M_\bR = M \otimes \bR$ and $N_\bR = N \otimes \bR$.

%%%%%%%%%%%%%%%%%%%%
%%%%%%%%%%%%%%%%%%%%
\subsection{Affine Toric Varieties}\label{afftoric_subsec}

Given a set $\sA = \{m_1, \ldots, m_s\} \subseteq M$, we get characters $\chi^{m_i} \colon T_N \to \bC^\ast$, the affine semigroup $\sfS \coloneqq \eZ \sA \subseteq M$ and the \emph{semigroup algebra} $\bC [\sfS] \coloneqq \bC [\chi^{m_1}, \ldots, \chi^{m_s}]$ with multiplication induced by the semigroup structure of $\sfS$. This gives rise to an affine toric variety $\Spec (\bC [\sfS])$, which is the Zariski closure of the image of $T_N \to \bC^s$ defined by characters $\chi^{m_i}$. In particular, the dimension of the affine toric variety is the rank of $\bZ \sA$.
%\cite[Proposition 1.1.8]{CLS}

\begin{example}
The affine semigroup $\eZ^s \subseteq \bZ^s$ gives the polynomial ring
\[
    \bC [\eZ^s] = \bC [x_1, \ldots, x_s],
\]
where $x_i =\chi^{e_i}$ and $\{e_1, \ldots, e_s\}$ is the standard basis of $\bZ^s$.
\end{example}

In what follows, we define $\bZ^{\sA} = \bigoplus_{i=1}^s\bZ e_{m_i} \cong \bZ^s$ and also write $\bC [\eZ^{\sA}] = \bC [\eZ^s]$ with $\chi^{e_{m_i}} = x_i$ for $\sA = \{m_1, \ldots, m_s\}$ when there is no danger of confusion. An inclusion $\sA \subseteq \sB$ induces a natural homomorphism $\bC [\eZ^{\sA}] \to \bC [\eZ^{\sB}]$ of polynomial rings. 

\begin{definition}
The \emph{toric ideal} $I_\sA$ of the affine toric variety $\Spec (\bC [\sfS])$ is defined by the kernel of the surjective $\bC$-algebra homomorphism $\bC [\eZ^{\sA}] \to \bC [\sfS]$ where $x_i \mapsto \chi^{m_i}$.
\end{definition}

We will use the convention in this paper that $m_\alpha$ denotes the lattice point $\sum_{i = 1}^r a_i m_i$ for any given set $\mathscr{S} = \{m_1, \ldots, m_r\} \subseteq M$ and $\alpha = (a_1, \ldots, a_r) \in \bZ^r = \bZ^{\mathscr{S}}$.
It induces a map of character lattices
\[
    \bZ^\sA = \bZ^s \to M
\]
that sends $\alpha$ to the lattice point $m_\alpha$. Let $L_\sA$ be defined by the following exact sequence
\begin{equation*}
    0 \to L_{\sA} \to \bZ^\sA \to M.
\end{equation*}
Then the toric ideal of $\Spec (\bC [\sfS])$ is the prime ideal%\cite[Proposition 1.1.9]{CLS}
\begin{equation}\label{tid_eqn}
    I_\sA = \langle x^\alpha - x^\beta \mid \alpha, \beta \in \eZ^s \mbox{ and } \alpha - \beta \in L_\sA  \rangle
\end{equation}
where $x^\gamma = x_1^{\gamma_1} x_2^{\gamma_2} \cdots x_s^{\gamma_s}$ for any $\gamma \in \eZ^s$.

\begin{example}
Given a rational polyhedral cone $\sigma \subseteq N_\bR$, the lattice points
\begin{equation*}
    \sfS_\sigma \coloneqq \sigma^\vee \cap M \subseteq M
\end{equation*}
form a semigroup. It is finitely generated by Gordan’s Lemma. Therefore this affine semigroup gives us an affine toric variety
\begin{equation*}
    U_\sigma \coloneqq \Spec (\bC [\sfS_\sigma]).
\end{equation*}
Furthermore, the cone $\sigma$ is strongly convex if and only if $\dim U_\sigma = \dim \sigma^\vee = \dim_\bR M_\bR$. %\cite[Theorem 1.2.18]{CLS}
\end{example}

%$\bC [\sfS] \simeq  \bC [\eZ^{\sA}] / I_{\sA}$

In fact, $\Spec (\bC [\sfS])$ is not necessarily normal.  If we regard $\sA$ as a generating set in $M$ and set $\sigma = \Cone (\sA)^\vee \subseteq N_\bR$.  By \cite[Proposition 1.3.8]{CLS}, 
$\sigma$ is a strongly convex rational polyhedral cone and the inclusion $\bC [\sfS] \subseteq \bC [\sfS_\sigma]$ induces a morphism $U_\sigma \to \Spec (\bC [\sfS])$ that is the normalization map of $\Spec (\bC [\sfS])$. It gives us an important fact. 
\begin{fact}\label{normalization}
 $\Spec (\bC [\sfS_\sigma])$ is the normalization of $\Spec (\bC [\sfS])$. Actually, the semigroup $\sfS_\sigma$ is the  \emph{saturation} of $\sfS$.
 \end{fact}

The following lemma is used in the proof of Theorem \ref{gfbir_prop}.

\begin{lemma}\label{easy_lem}
Let $\sigma$ be a strongly convex rational polyhedral cone in $N_\bR$. For any $m_1, m_2 \in M$, we can find a lattice point $m_0 \in \sfS_\sigma$ such that $m_0 + m_i \in \sfS_\sigma$ for $i = 1, 2$.   
\end{lemma}

\begin{proof}
Notice that $\dim \sigma^\vee = \dim_\bR M_\bR$ since $\sigma$ is strongly convex. Fix a lattice point $m'$ in the interior $\inte (\sigma^\vee)$ of $\sigma^\vee$. Since $m' + (1 / \ell) m_i \in \inte (\sigma^\vee)$ for a sufficiently large integer $\ell$,   $\ell m' + m_i \in \sfS_\sigma$ for $i = 1, 2$.
\end{proof}

%\begin{lemma}
%Let $\sigma$ be a strongly convex rational polyhedral cone in $N_\bR$.

%Let $\sA \subseteq M$ be a generating set of $\sfS_{\sigma}$.

%Then for any $\alpha_1, \alpha_2 \in \eZ^{\sA}$, there is a .... 
%
%\end{lemma}

%Before delving into algebra and geometry, we first present the following lemma as a simple observation. 

%Given two cone $\sigma$, $\sigma'\subseteq N$ contain in another cone $\sigma_0$. Let $\sA_0 \subseteq M$ be a generating set of $\sfS_{\sigma_0}$, and similarly for $\sA$ and $\sA'$. Since $\sigma^\vee \cap (\sigma')^\vee$ contains $\sigma_0^\vee$, we may assume that $\sA \cap \sA' \supseteq \sA_0$.

Assume that $\sigma$, $\sigma'$ and $\sigma_0$ are three rational polyhedral cones  in $N_\bR$ such that $\sigma_0$ contains $\sigma$ and $\sigma'$. Let $\sA_0 \subseteq M$ be a generating set of $\sfS_{\sigma_0}$ and similarly for $\sA$ and $\sA'$. Since $\sigma^\vee \cap (\sigma')^\vee$ contains $\sigma_0^\vee$, we may assume that $\sA \cap \sA' \supseteq \sA_0$. Consider the affine semigroup
\begin{equation}\label{splitsgp}
    \sfSp \coloneqq \eZ^{\sA_0}\oplus\eZ^{\sA\setminus\sA_0}\oplus\eZ^{\sA'\setminus\sA_0}.
\end{equation}
Via $\bC [\eZ^\sA]$ and $\bC [\eZ^{\sA'}]$ as subrings of $\bC [\sfSp]$, by abuse of notation, we still write $I_\sA$ and $I_{\sA'}$ for the corresponding ideals in $\bC [\sfSp]$. The following lemma characterizes binomials of $\bC [\sfSp]$ which belong to the sum $I_\sA + I_{\sA'}$ of toric ideals, which is used in the proof of Lemma \ref{n=3_reduced_equiv}.

\begin{lemma}\label{gfbir_lm}
For $\alpha_0$, $\beta_0 \in \eZ^{\sA}$, $\alpha$, $\beta \in \eZ^{\sA \setminus \sA_0}$ and $\alpha'$, $\beta' \in \eZ^{\sA' \setminus \sA_0}$, the binomial
\begin{equation}\label{gfbir_lm_eq}
    x^{\alpha_0 + \alpha + \alpha'} -x^{\beta_0 + \beta + \beta'} \in I_{\sA} + I_{\sA'} \subseteq \bC[\sfSp]
\end{equation}
if and only if there exists a sequence $\{(\gamma_{0i}, \gamma_i, \gamma_i')\}_{i = 0}^m \subseteq \sfSp$ such that
\[
    (\gamma_{01}, \gamma_1, \gamma_1') = (\alpha_0, \alpha, \alpha'), \quad (\gamma_{0m}, \gamma_m, \gamma_m') = (\beta_0, \beta, \beta'),
\]
and for each $1 \leq i \leq m-1$, either
\begin{align}
    \gamma_i = \gamma_{i+1}  &\mbox{ and }  (\gamma_{0i} + \gamma_i') -(\gamma_{0(i+1)} + \gamma_{i+1}') \in L_{\sA'}, \mbox{ or} \label{relation1}\\
    \gamma_i' = \gamma_{i+1}' &\mbox{ and } (\gamma_{0i} + \gamma_i) - (\gamma_{0(i+1)} + \gamma_{i+1}) \in L_{\sA} \label{relation2}
\end{align}
%\begin{equation}\label{relation1}
%    \gamma_i = \gamma_{i+1}  \mbox{ and }  (\gamma_{0i} + \gamma_i') -(\gamma_{0(i+1)} + \gamma_{i+1}') \in L_{\sA'}, \mbox{ or}
%\end{equation}
%or
%\begin{equation}\label{relation2}
%    \gamma_i' = \gamma_{i+1}' \mbox{ and } (\gamma_{0i} + \gamma_i) - (\gamma_{0(i+1)} + \gamma_{i+1}) \in L_{\sA}
%\end{equation}
holds.
\end{lemma}

%\begin{lm}\label{gfbir_lm}
%For $\alpha_0$, $\beta_0\in\eZ^{\sA}$, $\alpha$, $\beta\in\eZ^{\sA\setminus\sA_0}$, $\alpha'$, $\beta'\in\eZ^{\sA'\setminus\sA_0}$, the binomial
%\[x^{\alpha_0+\alpha+\alpha'}-x^{\beta_0+\beta+\beta'}\in I_{\sA}+I_{\sA'}\subseteq \bC[\eZ^{\sA_0}\oplus\eZ^{\sA\setminus\sA_0}\oplus\eZ^{\sA'\setminus\sA_0}]\]
%if and only if there exists a sequence $\{(\gamma_{0i},\gamma_i,\gamma_i')\}_{i=0}^m\subseteq\eZ^{\sA_0}\oplus\eZ^{\sA\setminus\sA_0}\oplus\eZ^{\sA'\setminus\sA_0}$ such that
%\[(\gamma_{01},\gamma_1,\gamma_1')=(\alpha_0,\alpha,\alpha'),\quad(\gamma_{0m},\gamma_m,\gamma_m')=(\beta_0,\beta,\beta'),\]
%and for each $1\leq i\leq m-1$, either
%\begin{equation}\label{relation1}
%    \gamma_i=\gamma_{i+1}\quad\text{and}\quad (\gamma_{0i}+\gamma_i')-(\gamma_{0(i+1)}+\gamma_{i+1}')\in L_{\sA'},
%\end{equation}
%or
%\begin{equation}\label{relation2}
%    \gamma_i'=\gamma_{i+1}'\quad\text{and}\quad (\gamma_{0i}+\gamma_i)-(\gamma_{0(i+1)}+\gamma_{i+1})\in L_{\sA}
%\end{equation}
%holds.
%\end{lm}

\begin{proof}
The necessary part is clear, since
\[x^{\gamma_{0i}+\gamma_i+\gamma'_i}-x^{\gamma_{0(i+1)}+\gamma_{i+1}+\gamma_{i+1}'}=\begin{cases}
x^{\gamma_i}(x^{\gamma_{0i}+\gamma_i'}-x^{\gamma_{0(i+1)}+\gamma_{i+1}'})\in I_{\sA'} & \text{if } \eqref{relation1} \text{ holds}\\
x^{\gamma_i'}(x^{\gamma_{0i}+\gamma_i}-x^{\gamma_{0(i+1)}+\gamma_{i+1}})\in I_{\sA} & \text{if } \eqref{relation2} \text{ holds}\\
\end{cases}\]
for each $1\leq i\leq m-1$.

%For the sufficient part, 
Conversely, suppose that \eqref{gfbir_lm_eq} holds. We write
\[x^{\alpha_0+\alpha+\alpha'}-x^{\beta_0+\beta+\beta'}=\sum_{k=1}^\ell c_k(x^{\gamma_k^+}-x^{\gamma_k^-}),\]
where $c_k\in\C$ and $\{(\gamma^+_k,\gamma^-_k)\}_{k=1}^\ell$ are distinct pairs satisfying \eqref{relation1} or \eqref{relation2} under the decomposition \eqref{splitsgp}. %$\Z^{\sA_0}\oplus\Z^{\sA\setminus\sA_0}\oplus\Z^{\sA'\setminus\sA_0}$ and.
Consider a weighed directed graph $G$ with vertices $\{\gamma_k^+,\gamma_k^-\mid 1\leq k\leq \ell\}$, directed edges from $\gamma_k^-$ to $\gamma_k^+$ with weight $\operatorname{wt}(\gamma_k^-,\gamma_k^+) = c_k$. Define the total degree at vertex $v$ by
\[\sum_{(u,v)\in E(G)} \operatorname{wt}(u,v)-\sum_{(v,u)\in E(G)}\operatorname{wt}(v,u).\]
%where $\operatorname{wt}(u,v)$ denoted the weight of the directed edge $(u,v)$. 
Then the total degree at $\alpha_0+\alpha+\alpha'$ is $1$, at $\beta_0+\beta+\beta'$ is $-1$, and at other vertices are $0$. Since the sum of total degrees over the vertices in a connected component of $G$ is zero, we conclude that 
$\alpha_0+\alpha+\alpha'$ and $\beta_0+\beta+\beta'$ are in the same connected component of $G$. Consequently, they can be connected by undirected edges, and the vertices along this path form the desired sequence, since the two vertices connected by the edge in $G$ satisfy \eqref{relation1} or \eqref{relation2}.
\end{proof}

%%%%%%%%%%%%%%%%%%%%
%%%%%%%%%%%%%%%%%%%%
\subsection{Fans}\label{fan_subsec}
%\subsection{Fans and Normal Toric Varieties}

For any fan $\Sigma$ in $N_{\bR}$, there is a corresponding normal toric variety $X_\Sigma$ on which the torus $T_N$ acts naturally. If a toric variety is normal, then it always comes from a fan $\Sigma$ in $N_\bR$ by Sumihiro's Theorem \cite{Sumihiro74}. The assignment
$\Sigma \mapsto X_\Sigma$ yields an equivalence of categories between the category of fans with morphisms of fans and the category of normal toric varieties with toric morphisms.

%$|\Sigma|$

%Recall that a map of fans $F \colon (N, \Sigma) \to (N', \Sigma')$ by definition is a $\bZ$-linear homomorphism from $\Sigma$ to $\Sigma'$, also denoted by $F$, such that for every cone $\sigma \in \Sigma$ there is a cone $\sigma' \in \Sigma'$ with $F_\bR (\sigma) \subseteq \sigma'$ (where $F_\bR \colon N_\bR \to N_\bR'$ is the scalar extension of $F$).

%We refer to any textbook on toric geometry such as [5, 8, 11] for details on this equivalence.

Given a fan $\Sigma_0$ in $N_\bR$, a fan $\Sigma$ \emph{refines} $\Sigma_0$ if every cone of $\Sigma$ is contained in a cone of $\Sigma_0$ and $|\Sigma| = |\Sigma_0|$. This yields a toric morphism $X_{\Sigma} \to X_{\Sigma_0}$ whose restriction map on $T_N$ is the identity map. 

\begin{definition}
Let $\Sigma$ and $\Sigma'$ be two refinements of $\Sigma_0$. We define the coarsest common refinement of $\Sigma$ and $\Sigma'$ over $\Sigma_0$ by
\begin{equation}\label{fp_fan}
    \tSigma \coloneqq \{ \sigma \cap \sigma' \mid \sigma\in\Sigma,\sigma'\in\Sigma' \mbox{ such that } \sigma, \sigma' \subseteq \sigma_0 \mbox{ for some } \sigma_0 \in \Sigma_0 \}
\end{equation}
in $N_\bR$.
\end{definition}

\begin{lm}\label{ccr_is_fan}
The coarsest common refinement $\widetilde{\Sigma}$ of $\Sigma$ and $\Sigma'$ over $\Sigma_0$ is a fan.
\end{lm}

\begin{proof}
%We will check three conditions for fan in \cite[Definition 3.1.2]{CLS}. 

To show that $\tSigma$ is a fan, first consider cones $\sigma\in\Sigma$ and $\sigma'\in\Sigma'$. Note that $\sigma\cap\sigma'$ is a rational polyhedral cone, since $\sigma$ and $\sigma'$ are the intersections of finitely many half-spaces in $N_\bR$. Additionally, $\sigma\cap\sigma'$ is strongly convex, since $\sigma$ is strongly convex. %For a face $\widetilde{\tau}$ of $\sigma\cap\sigma'$, say $\widetilde{\tau}=(\sigma\cap\sigma')\cap\Cone(m)^\perp$ for some $m\in(\sigma\cap\sigma')^\vee$. 

Next consider a face $\widetilde{\tau}$ of $\sigma\cap\sigma'$, say $\widetilde{\tau}=(\sigma\cap\sigma')\cap\Cone(m)^\perp$ for some $m\in(\sigma\cap\sigma')^\vee$. We can write $\widetilde{m}=m+m'$
under the decomposition
\begin{equation}\label{dual_cone_of_intersection}
(\sigma\cap\sigma')^\vee=(\sigma^{\vee\vee}\cap\sigma'^{\vee\vee})^\vee=(\sigma^\vee+\sigma'^\vee)^{\vee\vee}=\sigma^\vee+\sigma'^\vee.
\end{equation}
Then $\tau=\sigma\cap\Cone(z)^\perp$ and $\tau'=\sigma'\cap\Cone(z')^\perp$ are faces of $\sigma$ and $\sigma'$ respectively. Since $\gen{m,u},\gen{m',u}\geq 0$ for all $u\in\sigma\cap\sigma'$, the equation $\gen{\widetilde{m},u}=\gen{m,u}+\gen{m',u}$ implies that
\[\widetilde{\tau}=\tau\cap\tau'\in\widetilde{\Sigma}.\]
The same argument also implies that any face of $\sigma\cap\sigma'$ has the form $\tau\cap\tau'$, where $\tau$ and $\tau'$ are the faces of $\sigma$ and $\sigma'$ respectively. 
%By the same argument, we observe that any face of $\sigma\cap\sigma'$ has the form $\tau\cap\tau'$, where $\tau$ and $\tau'$ are the faces of $\sigma$ and $\sigma'$ respectively. Hence, we have
%\[(\sigma_1\cap\sigma_1')\cap(\sigma_2\cap\sigma_2')=(\sigma_1\cap\sigma_2)\cap(\sigma_1'\cap\sigma_2')\]
%is the face of $\sigma_i\cap\sigma_i'$ for $i=1$, $2$, since $\sigma_1\cap\sigma_2$ and $\sigma'_1\cap\sigma'_2$ are faces of $\sigma_i$ and $\sigma'$ respectively. So we conclude that $\widetilde{\Sigma}$ forms a fan in $N_\R$

Finally, we need to show that the intersection of any two cones $\sigma_1\cap\sigma_1'$ and $\sigma_2\cap\sigma_2'$ of $\tSigma$ is a face of each, where $\sigma_i \in \Sigma$ and $\sigma'_i \in \Sigma'$. Since $\sigma_1\cap\sigma_2$ and $\sigma'_1\cap\sigma'_2$ are faces of $\sigma_i$ and $\sigma'_i$ respectively, we deduce that 
\[(\sigma_1\cap\sigma_1')\cap(\sigma_2\cap\sigma_2')=(\sigma_1\cap\sigma_2)\cap(\sigma_1'\cap\sigma_2')\]
is the face of $\sigma_i\cap\sigma_i'$ for $i=1$, $2$. So we conclude that $\widetilde{\Sigma}$ forms a fan in $N_\R$
\end{proof}

We note that $\tSigma$ is the fiber product of $\Sigma$ and $\Sigma'$ over $\Sigma_0$ in the category of fans, so it is an important ingredient to study the fiber product of toric morphisms. 

% For the fiber product of general toric morphisms, we will introduce in the Appendix \ref{app_fiber_product}.

\subsection{Wall relations and Toric flips}\label{ToricFlips}

%%%%%%%%%%%%
In this section, we recall some basic results about toric flips. For details, please refer to \cite[Chapter 15]{CLS}, \cite[Chapter 14]{Matsuki} or the paper \cite{Reid83}. 
%\cite[Lemma 15.3.10, Theorem 15.3.13]{CLS}
%Proposition 15.4.5 (c) \cite{CLS}

%Let $\Sigma$ be a simplical fan in $N_\bR$ such that $X_\Sigma$ is semiprojective and let $\cR \subseteq \cNE(X_\Sigma)$ be an extremal ray of the Mori cone. 

%A \emph{flipping contraction} is a birational contraction whose exceptional locus has codimension $\geq 2$.

%\begin{theorem}\label{extflip_thm}
Let $X_\Sigma$ be a simplicial semiprojective toric variety and let $\cR \subseteq \cNE(X_\Sigma)$ be an extremal ray of its Mori cone. Reid showed that there is an extremal contraction
\[
    \phi_\cR \colon X_\Sigma \to X_{\Sigma_0}
\]
such that $\Sigma$ refines $\Sigma_0$ and $X_{\Sigma_0}$ is semiprojective. 

The idea of the construction for $\phi_\cR$ is briefly stated as follows. The extremal ray $\cR$  is generated by the curve class of an orbit closure $V (\tau)$ of an $(n-1)$-dimensional cone $\tau \in \Sigma$ and $\tau$ is called a \emph{wall}. Roughly speaking, $\Sigma_0$ 
%$\phi_\cR$ 
is obtained from the fan $\Sigma$ by "removing" all walls $\tau \in \Sigma$ with $[V (\tau)] \in \cR$.

To see the local picture, we pick a wall $\tau=\Cone(u_1,\ldots,u_{n-1})$ with $[V(\tau)]\in\cR$. Since $\Sigma$ is simplicial, the wall $\tau$ separates two $n$-dimensional cones
\begin{equation*}
    \begin{aligned}
        \sigma_{n+1} &= \Cone (u_1, \ldots, u_n),\\
        \sigma_n &= \Cone (u_1, \ldots, u_{n - 1}, u_{n + 1})
    \end{aligned}
\end{equation*}
in $\Sigma$. The primitive vectors $u_1, \ldots, u_{n + 1}$ of $\rho_1,\ldots,\rho_{n+1}$ in $\Sigma(1)$ span $N_\bR \simeq \bR^n$ and there is a nontrivial linear relation
\begin{equation} \label{wallrel_eqn}
    \sum_{i = 1}^{n + 1} b_i u_i = 0
\end{equation}
over $\bQ$, which is unique up to multiplication by a nonzero rational number and is called a \emph{wall relation}. Here, we may assume $b_{n+1}>0$. If $D_i = V(\rho_i)$ for $i=1,\ldots, n+1$, then we have that
\[D_i.V(\tau)=\dfrac{b_i}{b_{n+1}}D_{n+1}.V(\tau)\]
from the wall relation. For $\rho\in\Sigma(1)$ with its primitive vector $u_\rho\notin\{u_1,\ldots,u_{n+1}\}$, $D_\rho.\cR=0$ since $\rho$ and $\tau$ can not form a cone in $\Sigma$. We conclude  that the sets
\begin{equation}\label{int_R_nonzero}
    \begin{aligned}
        \{u_i\mid b_i<0\} & =\{u_\rho\mid \rho\in \Sigma(1), D_\rho.\cR<0\},\\
        \{u_i\mid b_i>0\} & =\{u_\rho\mid \rho\in \Sigma(1), D_\rho.\cR>0\}
    \end{aligned}
\end{equation}
are independent of the choice of the wall $\tau$ with $[V(\tau)]\in\cR$.

Furthermore, since $X_\Sigma$ is simplicial, we have the following exact sequence
\[0\to M\to \Z^{\Sigma(1)}\to \operatorname{Cl}(X_\Sigma)\to 0,\]
defined by $m\mapsto(\gen{m,u_\rho})_\rho$ and $e_\rho\mapsto D_\rho$. Its dual sequence is
\[0\to N_1(X_\Sigma)_\R\to\R^{\Sigma(1)}\to N_\R\to 0,\]
so we can identify $N_1(X_\Sigma)_\R$ with linear relations of primitive vectors in $\Sigma(1)$ and thus the wall relation \eqref{wallrel_eqn} is uniquely determined by the extremal ray $\cR$ up to multiplication by a nonzero rational number.

In \cite{Reid83}, Reid had shown that $\Sigma_0$ is a nonsimplicial fan if and only if $\phi_\cR$ is a small birational contraction. In this case, $|\{\rho\in\Sigma(1)\mid D_\rho.\cR<0\}|>1$ and 
\[\sigma_0\coloneqq \Cone(u_1,\ldots,u_n,u_{n+1})\]
is a non-simplicial $n$-dimensional cone in $\Sigma_0$. Also, there is a commutative diagram of birational toric morphisms
\begin{equation}\label{tflip_eqn}
\begin{tikzcd}
X_{\Sigma} \ar[rd,"\phi_\cR"'] & & X_{\Sigma'} \ar[ld,"\phi'"] \\
& X_{\Sigma_0} &
\end{tikzcd}    
\end{equation}
where $\Sigma'$ is given by another subdivision of $\Sigma_0$. The birational map 
\begin{equation}\label{tflipf_eqn}
   f \coloneqq (\phi')^{- 1} \circ \phi_\cR \colon X_{\Sigma} \dashrightarrow X_{\Sigma'}
\end{equation}
is called the \emph{toric flip} of $\phi_\cR$, and is called the \emph{toric flop} of $\phi_\cR$ if $K_{X_\Sigma}.\cR=0$.

To say more about the fan $\Sigma'$, for the wall $\tau$  as above,  we get the affine toric subvariety $U_{\sigma_0} \subseteq X_{\Sigma_0}$. From the wall relation \eqref{wallrel_eqn}, we define the sets
\begin{equation*}
    J_{-} = \{i \mid b_i < 0\} \mbox{, } J_{0} = \{i \mid b_i = 0\} \mbox{, } J_{+} = \{i \mid b_i > 0\}
\end{equation*}
and the cones %For a subset $J \subseteq \{1, \cdots, n + 1\}$, let 
\begin{equation*}
    \sigma_J = \Cone (u_i \mid i \in J) \mbox{ for } J \subseteq \{1, \ldots, n + 1\}.
\end{equation*}
%Then we define the following two sets of simplicial cones
%For simplicity, we assume $X_{\Sigma_0} = U_{\sigma_0}$. 
Via two subdivisions of $\sigma_0$, $\Sigma$ contains the cones $\{\sigma_J \mid J_{+} \not\subseteq J\}$ and $\Sigma'$ contains the cones $\{\sigma_J \mid J_{-} \not\subseteq J\}$. 
Also, the exceptional loci of $\phi_\cR$ and $\phi'$ over $U_{\sigma_0}$ are $V(\sigma_{J_{-}})$ and $V(\sigma_{J_{+}})$ which map onto $V(\sigma_{J_{-} \cup J_{+}})$, $\codim V(\sigma_{J_{\pm}}) = |J_{\pm}| \geq 2$ and $\dim V(\sigma_{J_{-} \cup J_{+}}) = |J_{0}|$.

As mentioned in \eqref{int_R_nonzero}, $\{u_i\mid i \in J_-\cup J_+\}$ is the set of primitive vectors $u_\rho$ of $\rho$ in $\Sigma(1)$ such that $D_\rho.\cR\neq 0$. Therefore, if $\sigma_{\operatorname{exc}}\coloneqq\Cone(u_\rho\mid \cR.D_\rho\neq 0)\in\Sigma_0$, then every $\sigma_0\in\Sigma_0(n)\setminus\Sigma(n)$ comes from
\[\operatorname{Star}(\sigma_{\operatorname{exc}})\coloneqq\{\sigma_0 \in \Sigma_0(n)\mid\sigma_{\operatorname{exc}}\prec\sigma_0\},\]
which is obtained by ``removing" walls. Hence,
\[\Exc(\phi_\cR)=V(\Cone(u_\rho\mid D_\rho.\cR<0))\quad \text{and} \quad \phi_\cR(\Exc(\phi_\cR))=V(\sigma_{\operatorname{exc}}).\]
In the subsequent sections, many problems can be checked locally, so we may fix $n+1$ vectors $u_1,\ldots,u_{n+1}$ from a wall $\tau$ together with the wall relation \eqref{wallrel_eqn}.

The terminology used for toric flips and toric flops coincides with the usual one in the minimal model program.

\begin{definition}\label{flip_MMP}
Let $(X,\Delta)$ be a log canonical pair. A projective morphism $\phi \colon X\to Z$ between normal varieties is a \emph{$(K_X+\Delta)$-flipping contraction} if
\begin{enumerate}[(1)]
    \item $X$ is $\Q$-factorial and $\Delta$ is an $\R$-divisor,
    \item $\phi$ is a small birational morphism of relative Picard number 1,
    \item $-(K_X+\Delta)$ is $\phi$-ample.
\end{enumerate}
A $(K_X+\Delta)$-flipping contraction is a \emph{flopping contraction} if $K_X$ is numerically relatively trivial.
\end{definition}

%By \cite[Theorem 14.3.3]{Matsuki}, for any extremal ray $\cR$, there exists a torus invariant boundary $\Q$-divisor $\Delta$, such that $(X_\Sigma,\Delta)$ is a klt pair and $(K_{X_\Sigma}+\Delta).\cR<0$. In the cases of $|J_-|>1$, $\phi_\cR$ is a small birational morphism of relative Picard number $1$. Hence, $\phi_\cR$ is a $(K_{X_\Sigma}+\Delta)$-flipping contraction and is a flopping contraction if $K_{X_\Sigma}.\cR=0$. 

\begin{remark}\label{rmk_toricflip}
Recall that, following \cite{Reid83}, a toric flip in this paper means a wall relation \eqref{wallrel_eqn} with $|J_-|>1$. Equivalently, it corresponds to the small birational contraction $\phi_{\cR}\colon X_{\Sigma}\to X_{\Sigma_0}$ associated with an extremal ray $\cR$ and the diagram \eqref{tflip_eqn}. 

Moreover, by \cite[Theorem 14-3-3]{Matsuki} (see also \cite[p.~781]{CLS}), there exists a torus invariant boundary $\Q$-divisor $\Delta$ such that $(X_\Sigma,\Delta)$ is klt and $(K_{X_\Sigma}+\Delta)\cdot \cR<0$. Hence $\phi_\cR$ is a $(K_{X_\Sigma}+\Delta)$-flipping contraction in the sense of Definition \ref{flip_MMP}, and it is a flopping contraction if and only if it gives a toric flop in our sense. Notably, this formulation does not require the condition $K_{X_\Sigma} \cdot \cR < 0$ typically found in the classical Minimal Model Program. Indeed, the sign of $K_{X_\Sigma} \cdot \cR$ has the same sign as $-\sum_{i=1}^{n+1} b_i$. %We say that a toric flip is a toric flop if $K_{X_{\Sigma}}. \cR = 0$.
\end{remark}

%$f = (\phi')^{- 1} \circ \phi$

%\begin{equation*}
%\begin{tikzcd}
%X_{\Sigma}|_{U_{\sigma_0}} \ar[rd,"\phi"'] \ar[rr, "f", dashed] & & X_{\Sigma'}|_{U_{\sigma_0}} \ar[ld,"\phi'"] \\
%& U_{\sigma_0} &
%\end{tikzcd},    
%\end{equation*}

%\begin{equation*}
%    \begin{aligned}
%        \Sigma_{-} = \{\sigma \mid \sigma \preceq \Cone (\nu_1, \cdots, \widehat{\nu_i}, \cdots, \nu_{n + 1}), i \in J_{+}\}, \\
%        \Sigma_{+} = \{\sigma \mid \sigma \preceq \Cone (\nu_1, \cdots, \widehat{\nu_i}, \cdots, \nu_{n + 1}), i \in J_{-}\}.
%    \end{aligned}
%\end{equation*}

%Furthermore, we always have $|J_{+}| \geq 2$.

%Then $\phi$ is flipping if and only if the number $|J_{-}| \geq 2$.

%a wall $\tau = \Cone (u_2, \cdots, u_n)$ of $\Sigma$.

%It is a flipping contraction if it is birational and its exceptional locus codimension $\geq 2$.

%Furthermore, we always have $|J_{+}| \geq 2$.

%Then $\phi$ is flipping if and only if $\Sigma_0$ is nonsimplicial if and only if the number $|J_{-}| \geq 2$.

%%%%%%%%

%By the toric cone theorem, the ray $\cR$ is generated by the class of an orbit closure $V (\tau)$ of a wall $\tau = \Cone (u_2, \cdots, u_n)$ of $\Sigma$. 

%%%%%%%%%%%%%%%%%%%%%%%%%%%%%%%%%%%%%%%%
%%%%%%%%%%%%%%%%%%%%%%%%%%%%%%%%%%%%%%%%
\section{Graph Closures}\label{gf_sec}

%Given a birational map $f \colon X_{\Sigma} \dashrightarrow X_{\Sigma'}$ induced by toric morphisms $X_{\Sigma} \rightarrow X_{\Sigma_0} \leftarrow X_{\Sigma'}$, we will investigate the graph closure $\Graphc$ of $f$, which is the closure of the fiber product $T_N \times_{T_N} T_N$ of tori, in this section.

In this section, we will investigate the graph closure $\Graphc$ of a toric birational map $f$ which is endowed with its reduced subscheme structure.

%We will investigate the graph closure $\Graphc$ of a toric birational map $f$. Theorem \ref{gfcri_thm} is the main result of this section.

%Let $\Sigma$ and $\Sigma'$ be two refinements of a fan $\Sigma_0$ and let $f \colon X_{\Sigma} \dashrightarrow X_{\Sigma'}$ be the birational map induced by $X_{\Sigma} \rightarrow X_{\Sigma_0} \leftarrow X_{\Sigma'}$. In this section, we will investigate the graph closure $\Graphc$ of $f$ which is the closure of the fiber product $T_N \times_{T_N} T_N$ of tori.

\begin{thm}\label{gfbir_prop}
Let $\Sigma$ and $\Sigma'$ be two refinements of a fan $\Sigma_0$ and let $f \colon X_{\Sigma} \dashrightarrow X_{\Sigma'}$ be the birational map induced by $X_{\Sigma} \rightarrow X_{\Sigma_0} \leftarrow X_{\Sigma'}$. Then we have:
\begin{enumerate}
    \item\label{gfbir_prop1} The graph closure $\Graphc$ is a (not necessarily normal) toric variety.

    \item\label{gfbir_prop2} The normalization of $\Graphc$ is the toric variety $X_{\tSigma}$ defined by the fan \eqref{fp_fan}.
\end{enumerate}
\end{thm}

In particular, the normalization of $\Graphc$ is the fiber product in the category of normal toric varieties for such a birational map $f$.

%Then the normalization of the graph closure of $f$ is the toric variety $X_{\tSigma}$ defined by the fan \eqref{fp_fan}.

\begin{proof}
The questions are local, so we fix a cone $\sigma \cap \sigma' \in \tSigma$ and the open set $U_{\sigma \cap \sigma'}$. By \eqref{fp_fan}, there is a cone $\sigma_0 \in \Sigma_0$ such that $\sigma$ and $\sigma'$ are contained in $\sigma_0$.  Let $\sA_0 \subseteq M$ be a generating set of $\sfS_{\sigma_0}$, and similarly for $\sA$ and $\sA'$. Since $\sigma^\vee \cap (\sigma')^\vee$ contains $\sigma_0^\vee$, we may assume that $\sA \cap \sA' \supseteq \sA_0$. %Now, we need separate to two cases.

%For the first case, if $\sA\cap\sA'=\sA_0$,
First, we deal with the simple case that $\sA\cap\sA'=\sA_0$. Set $\sB = \sA \cup \sA'$. According to $\sA \cap \sA' \supseteq \sA_0$, it follows that the fiber product $U_{\sigma} \times_{U_{\sigma_0}} U_{\sigma'}$ is the spectrum of the ring
\begin{equation}\label{gfbir_eqn}
    A \coloneqq \bC [\eZ^{\sA}] / I_\sA \otimes_{\bC [\eZ^{\sA_0}] / I_{\sA'}} \bC [\eZ^{\sA'}] / I_{\sA'} \simeq \bC [ \eZ^{\sB}] / I_{\sA} + I_{\sA'},
\end{equation} 
and $T_N \times_{T_N} T_N = \Spec A_a$ where $a \coloneqq \prod_{m \in \sA_0} x_m$. Denote by $\iota \colon A \to A_a$ the canonical ring homomorphism. Since $T_N \times_{T_N} T_N \simeq T_N$ is integral, the zero ideal $0_{A_a}$ is a radical ideal. Then the graph closure $\Graphc = \overline{T_N \times_{T_N} T_N}$ in $\Spec A$ is defined by the ideal
\[
    \bigcap_{a \notin \fp \in \Spec A} \fp = \iota^{-1} (\sqrt{0_{A_a}}) = \iota^{-1} (0_{A_a}).
\]

Note that the toric (prime) ideal $I_\sB$ contains $I_\sA + I_{\sA'}$. %We claim that $\Graphc \cap \Spec A = \bfV (I_\sB)$ in $\Spec A$. 
Then \eqref{gfbir_prop1} follows from the claim that $\Graphc \cap \Spec A = \bfV (I_\sB)$ in $\Spec A$. Indeed, using \eqref{tid_eqn} and the identification
\begin{equation}\label{gftid_eqn0}
    \bZ^{\sB} = \bZ^{\sA_0} \oplus \bZ^{\sA \setminus \sA_0} \oplus \bZ^{\sA' \setminus \sA_0}, 
\end{equation}
we pick
\begin{equation}\label{gftid_eqn1}
    (\alpha_0 + \alpha + \alpha') - (\beta_0 + \beta + \beta') \in L_\sB
\end{equation}
where $\alpha_0, \beta_0 \in \eZ^{\sA_0}$, $\alpha, \beta \in \eZ^{\sA \setminus \sA_0}$ and $\alpha', \beta' \in \eZ^{\sA' \setminus \sA_0}$. By Lemma \ref{easy_lem}, there is an element $\gamma_0 \in \eZ^{\sA_0}$ such that $m_{\gamma_0} + m_\alpha$ and $m_{\gamma_0} + m_\beta$ belong to  $\sfS_{\sigma_0}$, 
%\begin{equation*}
%    m_{\gamma_0} + m_\alpha, m_{\gamma_0} + m_\beta \in \sfS_{\sigma_0}, 
%\end{equation*}
so we can find $\alpha_1, \beta_1 \in \eZ^{\sA_0}$ such that
%\begin{equation*}
%    x^{\gamma_0 + \alpha} - x^{\alpha_1}, x^{\gamma_0 + \beta} - x^{\beta_1} \in I_{\sA_0}. 
%\end{equation*}
\begin{equation}\label{gftid_eqn2}
    (\gamma_0 + \alpha) - \alpha_1,\ (\gamma_0 + \beta) - \beta_1 \in L_{\sA_0}. 
\end{equation}
By \eqref{gftid_eqn1} and \eqref{gftid_eqn2}, we get 
\begin{equation*}
    (\alpha_0 + \alpha_1 + \alpha') - (\beta_0 + \beta_1 + \beta') \in L_{\sA'}
\end{equation*}
and thus the binomial
\begin{align*}
    x^{\gamma_0} (x^{\alpha_0 + \alpha + \alpha'} &- x^{\beta_0 + \beta + \beta'}) = \\ &\left[ x^{\alpha_0 + \alpha'} (x^{\gamma_0 + \alpha} - x^{\alpha_1}) - x^{\beta_0 + \beta'} (x^{\gamma_0 + \beta} - x^{\beta_1}) \right] + 
    \left[ x^{\alpha_0 + \alpha_1 + \alpha'} - x^{\beta_0 + \beta_1 + \beta'} \right]
\end{align*}
belongs to $I_\sA + I_{\sA'}$. Then the prime ideal $I_\sB / (I_\sA + I_{\sA'})$ defining $\bfV (I_\sB)$ is contained in the prime ideal $\iota^{-1}(0_{A_a})$ defining $\Graphc \cap \Spec A$. Therefore the claim follows from the dimension equality
\[
    \dim \Graphc = \rk N = \rk M = \dim \bfV (I_\sB).
\]
%that $\Graphc$ and $\bfV (I_\sB)$ have the same dimension $n = \rk M$.
According to 
\[
    \Cone (\sB)^\vee = (\sigma^\vee + (\sigma')^\vee )^\vee = \sigma \cap \sigma'
\]
and Fact \ref{normalization}, it follows that $U_{\sigma \cap \sigma'}$ is the normalization of
\[
   \bfV (I_\sB) \simeq \Spec (\bC [\eZ^{\sB}] / I_\sB),
\]
which proves \eqref{gfbir_prop2}.
%as required.

%For second case, 
In general, if $\sA\cap\sA'\supsetneq\sA_0$, we can modify the above proof as follows: The fiber product $U_{\sigma}\times_{U_{\sigma_0}}U_{\sigma'}$ is the spectrum of the ring
%\begin{equation}\label{gfbir_eqn2}
%    A = \bC [\eZ^{\sA}] / I_\sA \otimes_{\bC [\eZ^{\sA_0}] / I_{\sA'}} \bC [\eZ^{\sA'}] / I_{\sA'} \simeq \bC [ \eZ^{\sA_0}\oplus\eZ^{\sA\setminus\sA_0}\oplus\eZ^{\sA'\setminus\sA_0}] / I_{\sA} + I_{\sA'}.
%\end{equation}
\begin{equation}\label{gfbir_eqn2}
    A = \bC [\eZ^{\sA}] / I_\sA \otimes_{\bC [\eZ^{\sA_0}] / I_{\sA'}} \bC [\eZ^{\sA'}] / I_{\sA'} \simeq \bC [\sfSp] / I_{\sA} + I_{\sA'},
\end{equation}
where $\sfSp$ is the semigroup \eqref{splitsgp}. Let $L$ be the kernel of the map of lattices
%\[\begin{matrix}
%\Z^{\sA_0}\oplus\Z^{\sA\setminus\sA_0}\oplus\Z^{\sA'\setminus\sA_0} & \longrightarrow & M \\
%(\alpha_0,\alpha,\alpha') & \longmapsto & m_{\alpha_0}+m_\alpha+m_{\alpha'}
%\end{matrix},\]
\[
    \begin{matrix}
    \sfSp & \longrightarrow & M \\
    (\alpha_0,\alpha,\alpha') & \longmapsto & m_{\alpha_0}+m_\alpha+m_{\alpha'}
    \end{matrix},
\]
and let $I$ be the prime ideal of $\bC[\sfSp]$ %$\bC [ \eZ^{\sA_0}\oplus\eZ^{\sA\setminus\sA_0}\oplus\eZ^{\sA'\setminus\sA_0}]$ 
defined by
%\[I = \gen{x^\alpha-x^\beta\mid\alpha,\beta\in\eZ^{\sA_0}\oplus\eZ^{\sA\setminus\sA_0}\oplus\eZ^{\sA'\setminus\sA_0} \text{ and } \alpha-\beta\in L}.\]
\begin{equation}\label{gfbir_ideal}
    I = \gen{x^\alpha-x^\beta \mid \alpha, \beta \in \sfSp \mbox{ and } \alpha-\beta\in L}.
\end{equation}
%Then we claim that $\overline{\Gamma}_f\cap\Spec A=V(I)$, and the remaining parts of the proof from the first case also apply here.
Then we claim that $\overline{\Gamma}_f\cap\Spec A=V(I)$, and the same proof works for the general case.
\end{proof}

The following theorem gives us a combinatorial criterion for the normality of graph closures.

\begin{theorem}\label{gfcri_thm}
Let $f$ be as in Theorem \ref{gfbir_prop}. Then the following statements are equivalent:
\begin{enumerate}[(i)]
    \item\label{gfcri_thm1} $\overline{\Gamma}_f$ is normal.

    \item\label{gfcri_thm2} $\overline{\Gamma}_f$ is the toric variety of the fan $\tSigma$ defined in \eqref{fp_fan}.
    
    \item\label{gfcri_thm3} For all (maximal cone) $\sigma$ and $\sigma'$ contained in some cone of $\Sigma_0$,
    \begin{equation}\label{cone+cone=cone}
        \sfS_{\sigma} + \sfS_{\sigma'} = \sfS_{\sigma \cap \sigma'}.
    \end{equation}
\end{enumerate}
\end{theorem}

%\begin{thm}
%(old) Then the following statements are equivalent:
%\begin{enumerate}[(i)]
%    \item $X_{\widetilde\Sigma}\to X_{\Sigma}\times X_{\Sigma'}$ is a closed immersion. %(if and only if
%\begin{equation}\label{cone+cone=cone}
    %\left(\sigma^\vee\cap M\right)+\left(\sigma'^\vee\cap M\right)=\left(\sigma\cap\sigma'\right)^\vee\cap M 
%\end{equation}
%for all (maximal cone) $\sigma$, $\sigma'$ contained in some cone of $\Sigma_0$.),

 %   \item $X_{\widetilde\Sigma}\to\overline{\Gamma}_f$ is a closed immersion,

%    \item $X_{\widetilde\Sigma}\simeq\overline{\Gamma}_f$ is a toric variety,

%    \item $\overline{\Gamma}_f$ is normal.
%\end{enumerate}
%\end{thm}

%Consider the following three diagrams:
%\[\begin{tikzcd}[column sep = 15pt]
%& X_{\widetilde{\Sigma}} \arrow[d]\arrow[rdd,bend left,"\widetilde{f}'"]\arrow[ldd,bend right,"\widetilde{f}"']&&&& T_N \arrow[d]\arrow[rdd,bend left]\arrow[ldd,bend right]&\\
%& X_{\Sigma}\times_{X_{\Sigma_0}}X_{\Sigma'} \arrow[ld,"\pi"'] \arrow[rd,"\pi'"] &&&& T_N\times_{T_N}T_N \arrow[ld] \arrow[rd] &\\
%X_{\Sigma} \arrow[rd] && X_{\Sigma'} \arrow[ld] && T_N \arrow[rd] && T_N \arrow[ld] \\
%& X_{\Sigma_0}&&&& T_N&               
%\end{tikzcd}\]
%and
%\[\begin{tikzcd}[column sep = 15pt]
%& X_{\Sigma}\times_{X_{\Sigma_0}} X_{\Sigma'} \arrow[rd] &\\
%X_{\widetilde{\Sigma}} 
%\arrow[rr, "{(\widetilde{f},\widetilde{f}')}"] \arrow[ru] && X_{\Sigma}\times X_{\Sigma'}\\
%&&\\
%T_N \arrow[rr] \arrow[uu, hook]&& T_N\times_{T_N}T_N \arrow[uu, hook].
%\end{tikzcd}\]

\begin{proof}
The equivalence $\eqref{gfcri_thm1} \Leftrightarrow \eqref{gfcri_thm2}$ follows from Theorem \ref{gfbir_prop}. For $\eqref{gfcri_thm2} \Leftrightarrow \eqref{gfcri_thm3}$, we claim that \eqref{gfcri_thm3} is equivalent to saying that the natural morphism $X_{\tSigma} \to X_{\Sigma} \times X_{\Sigma'}$, called $\psi$, is a closed immersion. Indeed, the morphism $\psi$ is defined locally by
\begin{equation}\label{gfcri_eqn}
\begin{matrix}
\bC[\sfS_{\sigma}] \otimes_\bC \bC[\sfS_{\sigma'}] & \longrightarrow & \bC[\sfS_{\sigma \cap \sigma'}] \\
\chi^m \otimes \chi^{m'} & \longmapsto & \chi^{m + m'},
\end{matrix}
\end{equation}
and it is a closed immersion if and only if \eqref{gfcri_eqn} is surjective.

Since $\psi$ is proper, we have the birational morphism
\[
    \psi \colon X_{\tSigma} = \overline{T_N} \longrightarrow \overline{T_N \times_{T_N} T_N} = \Graphc
\]
and thus $\Graphc$ is the image of $\psi$. Therefore the closed immersion $\psi \colon X_{\tSigma} \to X_{\Sigma} \times X_{\Sigma'}$ gives the isomorphism $X_{\tSigma} \xrightarrow{\sim} \Graphc$. 
\end{proof}

As an application of Theorem \ref{gfcri_thm}, we are going to prove that $\Graphc$ is a normal toric variety for certain toric flips $f$. %for certain $3$-dimesional toric flips. %$f$ as in \eqref{tflipf_eqn}. We assume for simplicity that $X_{\Sigma_0} = U_{\sigma_0}$.

Let $f \colon X_\Sigma \dashrightarrow X_{\Sigma'}$ be a toric flip as in \eqref{tflipf_eqn}. Using the notation in Section \ref{ToricFlips}, we assume for simplicity that $X_{\Sigma_0} = U_{\sigma_0}$.

\begin{corollary}\label{gftoric_cor}
Let $f \colon X_\Sigma \dashrightarrow X_{\Sigma'}$ and $X_{\Sigma_0}$ be as above. Then $\Graphc$ is normal %the toric variety $X_{\tSigma}$ 
if $X_\Sigma$ satisfies the one of the following conditions.
\begin{enumerate}
    \item\label{gftoric_cor3D} Assume that $|J_+|=2$, say $J_{-}\cup J_+ = \{1, 2,\ldots,n-1\}$ and $J_{+} = \{n, n+1\}$. The cone $\sigma_{\{1,\ldots,n\}} \subseteq N_\bR \simeq \bR^n$ is smooth.

    \item\label{gftoric_cornDsm} The toric variety $X_{\Sigma}$ is smooth.
\end{enumerate}
\end{corollary}

\begin{proof}
To simplify the notation, set $\sigma_i = \sigma_{\{1, \ldots, n+1\} \setminus \{i\}}$. By Theorem~\ref{gfcri_thm}, it suffices to check \eqref{cone+cone=cone} holds when $(\sigma, \sigma') = (\sigma_j ,\sigma_i)$ for $i \in J_{-}$ and $j \in J_{+}$. The inclusion $\sfS_{\sigma \cap \sigma'} \supseteq \sfS_{\sigma} + \sfS_{\sigma'}$ follows directly from the general fact that $\sigma^\vee + (\sigma')^\vee = (\sigma \cap \sigma')^\vee$. Set
\begin{equation}\label{gftoric_eq}
    u \coloneqq - \sum_{i \in J_{-}} b_i u_i = \sum_{j \in J_{+}} b_j u_j.
\end{equation}
Then $\sigma \cap \sigma' = \Cone(u_k, u \mid k \notin \{i, j\})$.

First assume that \eqref{gftoric_cor3D} holds. Without loss of generality, we may assume that $i = 1$. If $j = n+1$, then the cone $(\sigma \cap \sigma')^\vee$ is generated by dual vectors
\[\{u_1^\vee,u_n^\vee\}\cup\{-b_1u_i^\vee+b_iu_1^\vee\}_{i\in J_-\setminus\{1\}}\cup\{u_j^\vee\}_{j\in J_0}.\]
Since the cone $\sigma_{\{u_1,\ldots,u_n\}}$ is smooth, we see that $\{u_1^\vee,\ldots, u_n^\vee\}$ forms a $\bZ$-basis of $M$. Hence
\begin{align*}
    \sfS_{\sigma \cap \sigma'} & = \Cone(u_1^\vee, -b_1u_i^\vee+b_iu_1^\vee,u_j^\vee\mid i\in J_-\setminus\{1\},j\in J_0) \cap M +\Cone(u_n^\vee) \cap M \\
        & \subseteq \sfS_{\sigma'} + \sfS_{\sigma}.
    \end{align*}

If $j = n$, by the linear relation \eqref{gftoric_eq}, we have
\[
    \overline{(\sigma \cap \sigma')^\vee \setminus \sigma'^\vee} = \Cone(u_2,\ldots,u_{n-1},u_{n+1},u,-u_n)^\vee = \Cone(u_2,\ldots,u_{n-1},u,-u_n)^\vee,
\]
which is generated by dual vectors
\[
    u_1^\vee,\ldots,u_{n-2}^\vee,\ -b_1u_{n-1}^\vee+b_{n-1}u_1^\vee,\ -u_n^\vee
\]
Using that $\{u_1^\vee, \ldots, u_n^\vee\}$ is a $\bZ$-basis of $M$ again, we get
\[
    \overline{(\sigma \cap \sigma')^\vee \setminus\sigma'^\vee}\cap M \subseteq \sfS_{\sigma'} + \sfS_{\sigma}.
\]
According to
\[
    (\sigma \cap \sigma')^\vee = \overline{(\sigma \cap \sigma')^\vee \setminus\sigma'^\vee}\cup (\sigma')^\vee,
\]
it follows that $\sfS_{\sigma \cap \sigma'} \subseteq \sfS_{\sigma} + \sfS_{\sigma'}$, which proves Corollary \ref{gftoric_cor} in the situation \eqref{gftoric_cor3D}.

Now assume that \eqref{gftoric_cornDsm} holds. For simplicity, we further assume that $1\in J_-$, $n+1 \in J_+$ and $(\sigma, \sigma') = (\sigma_{n+1}, \sigma_1)$. For $m \in \sfS_{\sigma \cap \sigma'}$, we let 
\begin{equation*}
    m_1 = \sum_{j\in J_0\cup J_+} \langle m, u_j \rangle u_j^\vee \quad\mbox{and}\quad m_2 = \sum_{i\in J_-} \langle m, u_i \rangle u_i^\vee.
\end{equation*}
According to $(\sigma \cap \sigma')^\vee = \Cone(u_k, u \mid k \notin \{1, n + 1\})^\vee$, it follows that $m_1 \in \sigma^\vee$ and $m_2 \in (\sigma')^\vee$. Hence $m = m_1 + m_2 \in \sfS_\sigma + \sfS_{\sigma'}$, as required.
\end{proof}

\begin{rmk}
While the condition \ref{gftoric_cor3D} in Corollary \ref{gftoric_cor} provides a criterion for general dimensions, the $3$-dimensional case (characterized by $|J_\pm|=2$ and $|J_0|=0$) admits a more refined description. Specifically, for general dimensions, condition \eqref{gftoric_cor3D} requires at least one maximal cone $\sigma_J \in \Sigma$ (where $J$ contains $J_-$) to be smooth. However, since the normality of $\overline{\Gamma}_f$ is symmetric with respect to $X_{\Sigma}$ and $X_{\Sigma'}$, the same conclusion holds if a maximal cone $\sigma_J \in \Sigma'$ (where $J$ contains $J_+$) is smooth. Combining these observations, we can strengthen the condition in $3$-dimensional as follows: $\overline{\Gamma}_f$ is normal provided that the cone $\sigma_J$ is smooth for some $J \subseteq \{1, 2, 3, 4\}$ with $|J| = 3$.
\end{rmk}

We provide two counterexamples to illustrate the necessity of the two conditions in Corollary \ref{gftoric_cor} (\ref{gftoric_cor3D}). These conditions require the smoothness of at least one maximal cone and that $|J_+|=2$. In each example, normality fails when exactly one of these requirements is violated.

%The following example illustrates that the existence of a smooth cone of $\Sigma$ is not a sufficient condition for the normality of $\Graphc$.

%Corollary~\ref{gftoric_cor} \eqref{gftoric_cor3D} may not hold for $\dim X_{\Sigma_0} = n >3$.

\begin{example}
Consider four integer vectors in $N=\Z^3$
\[u_1=(1,0,0),\quad u_2=(1,3,0),\quad u_3=(0,1,2),\quad u_4=(2,1,-3),\]
which satisfy the wall relation
\[-7u_1-5u_2+9u_3+6u_4=0.\]
Then we have $J_-=\{1,2\}$, $J_+=\{3,4\}$, $J_0=\varnothing$, and two simplicial fans $\Sigma$ and $\Sigma'$ defined as in Section \ref{fan_subsec}. Note that $\sigma_J$ is not smooth for any $J\subseteq\{1,2,3,4\}$ with $|J|=3$. Moreover, the graph closure of $f\colon X_{\Sigma}\dashrightarrow X_{\Sigma'}$ is not normal.

Indeed, we let cones $\sigma=\sigma_4$, $\sigma'=\sigma_1$, where $\sigma_i=\sigma_{\{1,\ldots,4\}\setminus\{i\}}$. We find that $\sfS_\sigma + \sfS_{\sigma'}$ has a generating set consisting of
\[
\begin{matrix}
(6, -2, 1) &
(3, -1, 1) &
(1, 0, 0) &
(0, 1, 0) \\
(0, 0, 1) & 
(-2, 2, -1) &
(-3, 3, -1) &
(-5, 4, -2).
\end{matrix}
\]
However, the lattice point $(-1, 1, 0) \in \sfS_{\sigma\cap\sigma'}$ but not in $\sfS_\sigma + \sfS_{\sigma'}$. Therefore \eqref{cone+cone=cone} does not hold for $(\sigma, \sigma')$.

This counterexample generalizes straightforwardly to higher dimensions. For $n > 3$, consider the embedding $\phi \colon \mathbb{Z}^3 \hookrightarrow N = \mathbb{Z}^n$ defined by mapping $\mathbb{Z}^3$ into the first three components of $\mathbb{Z}^n$. Consider the wall relation in $N$
\[-7u_1-5u_2+\sum_{j=3}^{n-1}0\cdot u_j+9u_n+6u_{n+1}=0,\]
where
\[u_1=\phi(1,0,0),\quad u_2=\phi(1,3,0),\quad u_n=\phi(0,1,2),\quad u_4=\phi(2,1,-3)\]
and $\{u_j\}_{j=3}^{n-1}$ are the remaining standard basis vector $\{e_4, \dots, e_n\}$. By a similar argument to the $3$-dimensional case, it follows that $\overline{\Gamma}_f$ is not normal.
\end{example}

\begin{example}\label{non-normal}
Let $\mathscr{S}_0 \coloneqq \{u_i \mid 1\leq i\leq 5\}$ be a $\bZ$-basis of $N \simeq \bZ^5$, and let $\sigma_0 = \Cone (\mathscr{S}_0)$. Consider the affine toric variety $X_{\Sigma_0} = U_{\sigma_0}$ and the wall relation
\[
    -3u_1-2u_2-u_3+3u_4+2u_5+u_6=0.
\]
Then we have $J_{-} = \{1, 2, 3\}$, $J_{+} = \{4, 5, 6\}$, $J_0 = \varnothing$, and two simplicial fans $\Sigma$ and $\Sigma'$ defined as in Section \ref{fan_subsec}.
Although the cone $\sigma_{J_{-} \cup \{4, 5\}}$ is smooth by construction, the graph closure of $f \colon X_{\Sigma} \dashrightarrow X_{\Sigma'}$ is not normal. 

Indeed, we let cones $\sigma = \sigma_4$, $\sigma' = \sigma_1$, where $\sigma_i = \sigma_{\{1, \ldots, 6\} \setminus \{i\}}$. 
We find that $\sfS_\sigma + \sfS_{\sigma'}$ has a generating set consisting of
\[
\begin{matrix}
( 2, 0, 0,  0, 3)&
( 0, 0, 1,  0, 0)&
( 0, 0, 2,  0, 1)&
(-1, 0, 3,  0, 0)\\
(-1, 1, 1,  0, 0)&
( 0, 3, 0,  2, 0)&
(-2, 3, 0,  0, 0)&
( 0, 0, 1, -1, 2)\\
( 0, 1, 0,  0, 0)&
( 1, 0, 0,  1, 0)&
( 0, 1, 0,  0, 1)&
( 1, 0, 0,  0, 1)\\
( 0, 0, 0, -1, 1)&
( 0, 0, 0, -1, 0)&
(-1, 2, 0,  0, 0).&  
\end{matrix}
\]

However, the lattice point $(-1, 2, 0, -1, 2) \in \sfS_{\sigma\cap\sigma'}$ but not in $\sfS_\sigma + \sfS_{\sigma'}$. Therefore \eqref{cone+cone=cone} does not hold for $(\sigma, \sigma')$.
\end{example}

\section{Fiber Products}\label{fp_sec} 
%\section{Fiber Products and Reducibility }\label{fp_sec} 

%Reducibility

%irreducibility
%In this section, we will use the same notation as in Section \ref{pre_sec} and study fiber products under toric flips.

In this section, we study the fiber product $X \coloneqq X_{\Sigma}\times_{X_{\Sigma_0}}X_{\Sigma'}$ in the category of schemes for two refinements $\Sigma$ and $\Sigma'$ of a fan $\Sigma_0$. In particular, we will concentrate on toric flips.%toric or flips.

\begin{proposition}\label{fpcri_prop}
Let $f \colon X_{\Sigma} \dashrightarrow X_{\Sigma'}$ be the birational map induced by toric morphisms $X_{\Sigma} \rightarrow X_{\Sigma_0} \leftarrow X_{\Sigma'}$. Then the following statements are equivalent:
\begin{enumerate}    
    \item\label{fpcri_prop0} The fiber product $X$ is a normal variety.
    
    \item\label{fpcri_prop1} The scheme $X$ is irreducible, reduced and the graph closure $\overline{\Gamma}_f$ is a normal toric variety.
    
    \item\label{fpcri_prop2} $X=X_{\widetilde{\Sigma}}$.
\end{enumerate}
\end{proposition}

%\begin{proposition}\label{fpcri_prop}
%Let $f \colon X_{\Sigma} \dashrightarrow X_{\Sigma'}$ be the birational map induced by toric morphisms $X_{\Sigma} \rightarrow X_{\Sigma_0} \leftarrow X_{\Sigma'}$. Then the following statements are equivalent:
%\begin{enumerate}
%    \item The fiber product $X$ is a toric variety.

%    \item The scheme $X$ is irreducible, reduced.

%    \item $X=\Graphc$.
%\end{enumerate}
%Moreover, by Theorem \ref{gfbir_prop}, the following statements are equivalent:

%\begin{enumerate}    
%    \item\label{fpcri_prop0} The fiber product $X$ is a normal toric variety.
    
%    \item\label{fpcri_prop1} The scheme $X$ is irreducible, reduced and the graph closure $\overline{\Gamma}_f$ is a normal toric variety.
    
%    \item\label{fpcri_prop2} $X=X_{\widetilde{\Sigma}}$.
%\end{enumerate}
%\end{proposition}

Note that the conditions in \eqref{fpcri_prop1}  are all local properties by Theorem \ref{gfcri_thm}.

%\begin{proof}
%The $X_\Sigma$ is irreducible and so is the graph closure $\overline{T_N \times_{T_N} T_N } = \Graphc$. If the fiber product $X$ is irreducible and reduced, then $\Graphc = X_{\red} = X$. Then Proposition \ref{fpcri_prop} follows from Theorem \ref{gfbir_prop}.
%\end{proof}

\begin{proof}
Note that $X_\Sigma$ is irreducible and so is the graph closure $\overline{T_N \times_{T_N} T_N } = \Graphc$. If the fiber product $X$ satisfies one of the conditions \eqref{fpcri_prop0}, \eqref{fpcri_prop1}, \eqref{fpcri_prop2}, then $X$ is irreducible and reduced and thus $\Graphc = X = X_{\red} $. Therefore Proposition \ref{fpcri_prop} follows from Theorem \ref{gfbir_prop}.
\end{proof}

\begin{example}
In this example, we will illustrate how Proposition \ref{fpcri_prop} provides a method to construct a non-normal fiber product.
Let $\sigma_0=\Cone(u_1,\ldots,u_5)$ be a strongly convex cone in $N\simeq \Z^3$ such that $\Cone(u_i,u_j)$ is a $2$-dimensional face if and only if $|i-j|=1$ or $4$. To simplify the notation, let
\[\sigma_J=\Cone(u_j\mid j\in J) \text{ for } J\subseteq\{1,2,3,4,5\}.\]
We start with the $3$-dimensional toric variety $X_\Sigma$ such that $\Sigma$ contains three cones $\sigma_{\{1,2,4\}}$, $\sigma_{\{2,3,4\}}$, $\sigma_{\{1,4,5\}}$. If we perform two consecutive toric flips along the walls $\sigma_{\{2,4\}}$ and $\sigma_{\{1,4\}}$, then we will obtain the toric variety $X_{\Sigma'}$. Let $\Sigma_0$ be the fan removing the walls $\sigma_{\{2,4\}}$ and $\sigma_{\{1,4\}}$ from $\Sigma$. Then $X_{\Sigma}$ and $X_{\Sigma'}$ are two varieties over $X_{\Sigma_0}$. We claim that the fiber product $X=X_{\Sigma}\times_{X_{\Sigma_0}}X_{\Sigma'}$ is not a normal toric variety.

According to Proposition \ref{fpcri_prop}, if $X$ is a normal toric variety, then $X=X_{\widetilde{\Sigma}}$, where $\widetilde{\Sigma}$ is defined in \eqref{fp_fan}. Let $\phi$, $\phi'$, and $\widetilde{\phi}$ be the birational morphisms from $X_{\Sigma}$, $X_{\Sigma'}$ and $X_{\widetilde{\Sigma}}$ to $X_{\Sigma_0}$ respectively. Observe that the exceptional locus $Z$ of $\phi$ is 
\[\operatorname{Exc}(\phi)=V(\sigma_{\{1,4\}})\cup V(\sigma_{\{2,4\}})=\P^1\cup_{\text{pt}}\P^1,\]
which maps to $S\coloneqq V(\sigma_0)$. Similarly, the exceptional locus $Z'$ of $\phi'$ is $\P^1\cup_{\text{pt}}\P^1$. Then
\begin{equation}\label{too_many_exceptional}
(\P^1\cup_{\text{pt}}\P^1)\times(\P^1\cup_{\text{pt}}\P^1)\subseteq X=X_{\widetilde{\Sigma}}
\end{equation}
contained in $\operatorname{Exc}(\widetilde{\phi})$. Since $\operatorname{Exc}(\widetilde{\phi})$ is the union of three $\P^1\times\P^1$, it can be contained in at most three disjoint $(\C^\times)^2$. However, \eqref{too_many_exceptional} indicates that $\operatorname{Exc}(\widetilde{\phi})$ contains at least four disjoint $(\C^\times)^2$, which leads to a contradiction. So we conclude that $X$ is not a normal toric variety.
\end{example}

%%%%%%%%
%%%%%%%%
%\begin{example}
%    blow-up, not irreducible, $X_{\Sigma_0} = \bP^2$, $X_{\Sigma} = X_{\Sigma'}$ is the blow-up of $\bP^2$ at a point
%\end{example}
%%%%%%%%
%%%%%%%%

%\begin{theorem}
%If $X_{\Sigma} \rightarrow X_{\Sigma_0} \leftarrow X_{\Sigma'}$ is a toric flip \eqref{tflip_eqn}, then the reduced scheme $X_{\red}$ associated to the fiber product $X = X_{\Sigma} \times_{X_{\Sigma_0}} X_{\Sigma'}$ is a (not necessarily normal) toric variety. Moreover, the normalization of $X_{\red}$ is the toric variety of the fan $\tSigma$ defined in \eqref{fp_fan}.
%\end{theorem}

%Given a toric flip $f \colon X_{\Sigma} \dashrightarrow X_{\Sigma'}$ via $X_{\Sigma} \rightarrow X_{\Sigma_0} \leftarrow X_{\Sigma'}$ \eqref{tflip_eqn},....

%%%%%%%%%%%%%%%%%%%%
%%%%%%%%%%%%%%%%%%%%
\subsection{Irreducibility}

We are going to prove that the reduced scheme associated to the fiber product under a toric flip is a toric variety.

\begin{theorem}\label{fpred_thm}
%Let $X_{\Sigma} \rightarrow X_{\Sigma_0} \leftarrow X_{\Sigma'}$ be a toric flipping contractions as in \eqref{tflip_eqn} 
Let $f \colon X_{\Sigma} \dashrightarrow X_{\Sigma'}$ be a toric flip via toric morphisms $X_{\Sigma} \rightarrow X_{\Sigma_0} \leftarrow X_{\Sigma'}$ \eqref{tflip_eqn} and $X = X_{\Sigma} \times_{X_{\Sigma_0}} X_{\Sigma'}$ be the fiber product. Then we have:
\begin{enumerate}
    \item\label{fpred_thm1} The reduced scheme $X_{\red}$ associated to $X$ is the toric variety $\Graphc$.

    \item\label{fpred_thm2} The normalization of $X_{\red}$ is the toric variety $X_{\tSigma}$ defined by the fan \eqref{fp_fan}.
\end{enumerate}
\end{theorem}

\begin{proof}
If $X_{\red}$ is irreducible, then Theorem \ref{fpred_thm} follows from Theorem \ref{gfbir_prop} and $\Graphc = X_{\red}$. As this is a local question, we may assume that $X_{\Sigma_0}$ is an affine toric variety $U_{\sigma_0}$ defined by the cone $\sigma_0 = \Cone (u_1, \ldots, u_{n + 1})$ and there is a wall relation \eqref{wallrel_eqn}
\begin{equation}\label{fpred_wall}
    u \coloneqq - \sum_{i \in J_{-}} b_i u_i = \sum_{j \in J_{+}} b_j u_j.
\end{equation}

%To simplify the notation, set $\sigma_k = \sigma_{\{1, \cdots, n+1\} \setminus \{k\}}$.
% and $\sigma_{k, \ell} = \sigma_{\{1, \cdots, n+1\} \setminus \{k, \ell\}}$ for $1 \leq k < \ell \leq n + 1$.
 %Set $J_-^* = J_- \setminus \{1\}$ and $J_+^* = J_+ \setminus \{n + 1\}$.
Without loss of generality, we can assume that $1 \in J_{-}$ and $ n + 1 \in J_{+}$. Set 
\[
    \sigma = \sigma_{\{1, \ldots, n+1\} \setminus \{n + 1\}} \mbox{ and } \sigma' = \sigma_{\{1, \ldots, n+1\} \setminus \{1\}}.
\]
%By symmetry, it suffices to show that $U_{\sigma_{n + 1}} \times_{U_{\sigma_0}} U_{\sigma_1}$ is irreducible. 
To prove \eqref{fpred_thm1}, it suffices to show that the reduced scheme associated to $U_{\sigma} \times_{U_{\sigma_0}} U_{\sigma'}$ is an affine toric variety. 

Let $\sA_0, \sA$ and $\sA'\subseteq M$ be generating sets of $\sfS_{\sigma_0}$, $\sfS_{\sigma}$ and $\sfS_{\sigma'}$ respectively. We may assume that $\sA \cap \sA' = \sA_0$ because of the inclusion $\sigma_0^\vee \subseteq \sigma^\vee \cap (\sigma')^\vee$. Set $\sB = \sA \cup \sA'$. Clearly, the fiber product $U_{\sigma} \times_{U_{\sigma_0}} U_{\sigma'}$ is the spectrum of the ring \eqref{gfbir_eqn} as we have seen in the proof of Theorem \ref{gfbir_prop}. Then the irreducibility of $X_{\red}$ follows from the claim that
\begin{equation}\label{fpred_eqn1}
    \sqrt{I_{\sA} + I_{\sA'}} = I_{\sB}.
\end{equation}
%and \eqref{fpred_thm2} follows from $\Cone (\sB)^\vee = \sigma \cap \sigma'$ and Proposition \ref{normali_prop}.

%We claim that the radical ideal $\sqrt{I_{\sA} + I_{\sA'}}$ is the toric ideal $I_{\sB}$.
To see \eqref{fpred_eqn1}, the inclusion $\sqrt{I_{\sA} + I_{\sA'}} \subseteq I_{\sB}$ follows from the fact that $I_\sB$ is a prime ideal and $I_{\sA} + I_{\sA'} \subseteq I_{\sB}$.
For the other inclusion, the proof will be divided into four steps.

\noindent\textbf{Step 1.} As in the proof of Theorem \ref{gfbir_prop}, we pick
\begin{equation}\label{fpred_eqn2}
    (\alpha_{01} + \alpha_1 + \alpha'_1) - (\alpha_{02} + \alpha_2 + \alpha'_2) \in L_\sB
\end{equation}
where $\alpha_{01}, \alpha_{02} \in \eZ^{\sA_0}$, $\alpha_1, \alpha_2 \in \eZ^{\sA \setminus \sA_0}$ and $\alpha'_1, \alpha'_2 \in \eZ^{\sA' \setminus \sA_0}$, using the identification \eqref{gftid_eqn0}. To simplify notation, we set $\beta_1 = \alpha_{01} + \alpha_1 + \alpha'_1$ and $\beta_2 = \alpha_{02} + \alpha_2 + \alpha'_2$. 
Suppose that there exists $k_i \in \bN$ such that
%\[
%    x^{k_i \alpha_i} (x^{\alpha_0 + \alpha_1 + \alpha'_1} - x^{\beta_0 + \alpha_2 + \alpha'_2}) \in \sqrt{I_\sA + I_{\sA'}}
%\]
\begin{equation}\label{fpred_eqn3}
    x^{k_i \alpha_{0i}} (x^{\beta_1} - x^{\beta_2}) \in \sqrt{I_\sA + I_{\sA'}}
\end{equation}
for $i=1$, $2$. Using binomial expansion
\[
    (x^{\beta_1} - x^{\beta_2})^{k_1 + k_2 + 1} = \sum_{k=0}^{k_1+k_2} \binom{k_1+k_2}{k} x^{k \beta_1} x^{(k_1 + k_2 - k) \beta_2} (x^{\beta_1} - x^{\beta_2})
\]
and that either $k \geq k_1$ or $k_1+k_2-k \geq k_2$ holds, we get $x^{\beta_1} - x^{\beta_2} \in \sqrt{I_\sA + I_{\sA'}}$ and thus verify \eqref{fpred_eqn1}, as required.

To get \eqref{fpred_eqn3}, we can replace, if necessary, \eqref{fpred_eqn2} by the following two relations
\begin{align}
    (k_1 + 1) \alpha_{01} + \alpha_1 + \alpha'_1 \equiv (k_1 \alpha_{01} + \alpha_{02})+ \alpha_2 + \alpha'_2 &\pmod{L_\sB}, \label{fpred_eqn4} \\ 
    (k_2 \alpha_{02} + \alpha_{01}) + \alpha_1 + \alpha'_1 \equiv (k_2 + 1) \alpha_{02} + \alpha_2 + \alpha'_2 &\pmod{L_\sB}. \label{fpred_eqn5} 
\end{align}

In general, the pair $(\alpha_{01}, \alpha_{02})$ can be replaced by any pair in $\{ \alpha_{01}, \alpha_1, \alpha'_1 \}\times\{\alpha_{02}, \alpha_2, \alpha'_2\}$. This process of replacements will occur frequently in the following algorithm of finding such pair $(k_1,k_2)$ of natural numbers.

\noindent\textbf{Step~2.} To simplify notation, we set $m_{0i} = m_{\alpha_{0i}}$, $m_i = m_{\alpha_i}$ and $m'_i = m_{\alpha'_i}$ for $i = 1, 2$. If $m_{01}, m_{02} \in \sigma_0^\vee\setminus\Cone(u_1)^\perp$, then we can take $k_1$, $k_2 \in \bN$ such that $m_{0i} + (1 / k_i) m'_j \in \sigma_0^\vee$ for each $j = 1,2$. Set $m_{ij} = k_i m_{0i} + m'_j$ for $1 \leq i, j \leq 2$. We find that $m_{ij} \in \sigma_0^\vee \cap M = \sfS_{\sigma_0}$ and thus $m_{ij} = m_{\beta_{ij}}$ for some $\beta_{ij} \in \eZ^{\sA_0}$. Then the relation \eqref{fpred_eqn4} reduces to
\begin{align*}
    k_1 \alpha_{01} + \alpha'_1 \equiv \beta_{11} &\pmod{L_{\sA'}}, \\
    \beta_{12} \equiv k_1 \alpha_{01} + \alpha'_2 &\pmod{L_{\sA'}}, \\
    \beta_{11} + \alpha_{01} + \alpha_1 \equiv \beta_{12} + \alpha_{02} + \alpha_2 
    &\pmod{L_{\sA}}.
\end{align*}
Hence, the binomial corresponding to the relation \eqref{fpred_eqn4} is in $I_\sA + I_{\sA'}$. A similar argument holds for \eqref{fpred_eqn5}.

We also can replace the condition $\sigma_0^\vee\setminus\Cone(u_1)^\perp$ with $\sigma_0^\vee\setminus\Cone(u_{n+1})^\perp$. Then we can take $k_1$, $k_2\in\N$ such that $m_{0i}+(1/k_i)m_j\in\sigma_0^\vee$ for each $j=1$, $2$, and the similar argument also works.

\noindent\textbf{Step~3.} 
We use the following notation for the remainder of the proof. Given a face $\tau$ of a cone $\sigma$, we define $\tau^* = \sigma^\vee \cap \tau^{\perp}$, the dual face of $\tau$. Then $\tau^*$ is a face of $\sigma^\vee$. %\cite[Proposition 1.2.10]{CLS}

Consider the face $\sigma_{J_0 \cup J_{-}}^*$ of $\sigma^\vee$ and the face $\sigma_{J_0 \cup J_{+}}^*$ of $(\sigma')^\vee$. In the following cases, we can immediately check that:
%\begin{enumerate}[\lb]
%    \item If $x_i\neq 0$, then there is $k_i\in\N$ s.t. $k_ix_i=x_i'+z_i'$, where $x_i'\in\sigma_0^\vee\setminus(\partial\sigma^\vee\cap\sigma'^\vee)$, $z_i\in \sigma_{J_0 \cup J_{+}}^*$.

%    \item If $y_i\notin F$, then there is $k_i\in\N$ s.t. $k_iy_i=x_i'+y_i'$, where $x_i'\in\sigma_0^\vee\setminus(\partial\sigma^\vee\cap\sigma'^\vee)$, $y_i\in \sigma_{J_0 \cup J_{-}}^*$.

%    \item If $z_i\notin F'$, then there is $k_i\in\N$ s.t. $k_iz_i=x_i'+z_i'$, where $x_i'\in\sigma_0^\vee\setminus(\partial\sigma^\vee\cap\sigma'^\vee)$, $z_i\in \sigma_{J_0 \cup J_{+}}^*$.
%\end{enumerate}
\begin{enumerate}[\lb]
    \item If $m_{0i} \neq 0$, then $k_i m_{0i} \in [\sigma_0^ \vee \setminus\Cone(u_1)^\perp] + \sigma_{J_0 \cup J_{+}}^*$ for some $k_i\in\N$.

    \item If $m_i \notin \sigma_{J_0 \cup J_{-}}^*$, then $k_i m_i \in [\sigma_0^ \vee \setminus\Cone(u_{n+1})^\perp] + \sigma_{J_0 \cup J_{-}}^*$ for some $k_i\in\N$.

    \item If $m'_i\notin \sigma_{J_0 \cup J_{+}}^*$, then $k_i m'_i \in [\sigma_0^ \vee \setminus\Cone(u_1)^\perp] + \sigma_{J_0 \cup J_{+}}^*$ for some $k_i\in\N$.
\end{enumerate}

We are going to verify \eqref{fpred_eqn3} in the above cases. Suppose that one of the above three conditions holds for $i=1,2$, and let $\gamma_i$ denote the vector corresponding to the satisfied lattice point, where $\gamma_i \in \{\alpha_{0i}, \alpha_i, \alpha'_i \}$. Replace \eqref{fpred_eqn2} with \eqref{fpred_eqn4} and \eqref{fpred_eqn5} on the pairs $(\gamma_1, \gamma_2)$ and $(k_1,k_2)$ as in Step $1$. By replacing $k_i \gamma_i$ with the above expression on the both sides of \eqref{fpred_eqn4} and \eqref{fpred_eqn5}, we can reduce it to Step~2.

For example, we treat the case $\gamma_1=\alpha_{01}$ and $\gamma_2=\alpha_{02}$, which is one of the nine cases. We take
\[
    m_{\gamma_{0i}} \in [\sigma_0^\vee\setminus\Cone(u_1)^\perp] \quad\mbox{ and } \quad m_{\gamma_i'} \in \sigma_{J_0\cup J_+}^* 
\]
for some $\gamma_{0i}\in\Z_{\geq 0}^{\sA_0}$, $\gamma_i'\in\Z_{\geq 0}^{\sA'\setminus\sA_0}$ such that
$k_im_{0i}=m_{\gamma_{0i}}+m_{\gamma_i'}$. Then the relation \eqref{fpred_eqn4} reduces to the congruences $k_1\alpha_{01} \equiv \gamma_{01}+\gamma_1' \pmod{L_{\sA'}}$ and 
\begin{equation}\label{fpred_eqn2_new}
    (\gamma_{01}+\alpha_{01}) + \alpha_1 + (\alpha_1'+\gamma_1') \equiv  (\gamma_{01}+\alpha_{02})+ \alpha_{02} + (\alpha_2'+\gamma_1') 
    \pmod{L_{\sB}}.
\end{equation}
%\begin{align}
%    k_1\alpha_{01}\equiv \gamma_{01}+\gamma_1'&\pmod{L_{\sA'}},\notag \\
%    \gamma_{01}+\gamma_1'\equiv k_1\alpha_{01}&\pmod{L_{\sA'}}\notag, \\
%    (\gamma_{01}+\alpha_{01}) + \alpha_1 + (\alpha_1'+\gamma_1') \equiv  (\gamma_{01}+\alpha_{02})+ \alpha_{02} + (\alpha_2'+\gamma_1') 
%    &\pmod{L_{\sB}}.\label{fpred_eqn2_new}
%\end{align}
Note that $m_{\gamma_{01}+\alpha_{01}}$, $m_{\gamma_{01}+\gamma_{02}}\in\sigma_0^\vee\setminus\Cone(u_1)^\perp$, since $m_{\gamma_{01}}\in\sigma_0^\vee\setminus\Cone(u_1)^\perp$. So we can apply Step $2$ to show that the binomial corresponding to the relation \ref{fpred_eqn2_new} is in $\sqrt{I_\sA+I_{\sA'}}$, and thus the binomial corresponding to the relation \eqref{fpred_eqn4} is in $\sqrt{I_\sA+I_{\sA'}}$. A similar argument holds for \eqref{fpred_eqn5}.

\noindent\textbf{Step 4.} The remaining case to consider is when at least one of $i \in \{1, 2\}$ does not satisfy the condition in Step~$3$. Without loss of generality, we assume that $i=2$, that is,
\[
    m_{02} = 0,\ m_2 \in \sigma_{J_0 \cup J_{-}}^*,\ m'_2 \in \sigma_{J_0 \cup J_{+}}^*.
\]
Recall that $u$ is the vector \eqref{fpred_wall} associated to the wall relation and thus $\Cone(u)$ is a face of $\sigma \cap \sigma'$. Since 
\begin{equation}\label{fpred_eqn6}
    \sigma_{J_0 \cup J_{-}}^* + \sigma_{J_0 \cup J_{+}}^* =(\sigma \cap \sigma')^\vee \cap \Cone(u)^\perp
\end{equation}
is a face of $(\sigma\cap\sigma')^\vee$ and $m_{01} + m_1 + m'_1 = m_2 + m'_2$ belongs to the left hand side of \eqref{fpred_eqn6}, we get
\[
    m_{01} \in \sfS_{\sigma_0} \cap (\sigma_{J_0 \cup J_{-}}^* + \sigma_{J_0 \cup J_{+}}^*) = \{ 0 \},\ m_1 \in \sigma_{J_0 \cup J_{-}}^*,\ m'_1 \in \sigma_{J_0 \cup J_{+}}^*.
\]
Therefore we find that
\[
    m_1 - m_2 = m'_1 - m'_2 \in \sigma_{J_0 \cup J_{-}}^\perp \cap \sigma_{J_0 \cup J_{+}}^\perp = \{ 0 \}.
\]
Since $\sigma_0^\vee$ is strongly convex, it forces $\alpha_{01}=\alpha_{02}=0$, and thus
\begin{align*}
    \alpha_1 \equiv \alpha_2 &\pmod{L_{\sA}}, \\
    \alpha_{01} + \alpha'_1 \equiv \alpha_{02} + \alpha'_2 
    &\pmod{L_{\sA'}}.
\end{align*}
Hence the binomial
\begin{align*}
    x^{\beta_1} - x^{\beta_2} = x^{\alpha_{01} + \alpha'_1} (x^{\alpha_1} - x^{\alpha_2}) + x^{\alpha_2} (x^{\alpha_{01} +\alpha'_1} - x^{\alpha_{02} + \alpha'_2})
\end{align*}
belongs to $I_\sA + I_{\sA'}$, as required.

The second statement follows from Theorem \ref{gfbir_prop}.
\end{proof}

\begin{remark}\label{gfbir_rmk}
For any (not necessarily maximal) cone $\sigma\in\Sigma$ and $\sigma'\in\Sigma'$ contained in a cone $\sigma_0\in\Sigma_0$, we also have the equality 
\begin{equation}\label{gfbir_rmk_eqn1}
    \sqrt{I_\sA + I_{\sA'}} = I
\end{equation}
where $I$ is the ideal  of $\bC[\sfSp]$ as defined in \eqref{gfbir_ideal} and $I_\sA$ and $I_{\sA'}$ are corresponding toric ideals of $\sfS_\sigma$ and $\sfS_{\sigma'}$ respectively. Indeed, by Theorem \ref{gfbir_prop} and Proposition \ref{fpcri_prop}, we have
\[\Spec (\bC[\sfSp] / \sqrt{I_\sA+I_{\sA'}}) =(U_\sigma\times_{U_{\sigma_0}}U_{\sigma'})_{\operatorname{red}}=\overline{\Gamma}_f\cap(U_\sigma\times_{U_{\sigma_0}}U_{\sigma'})= \Spec \left(\bC[\sfSp] / I\right)\]
%\[\Spec\quotient{\C[\sfS_{\operatorname{split}}]}{\vphantom{\sum}\sqrt{I_\sA+I_{\sA'}}}=(U_\sigma\times_{U_{\sigma_0}}U_{\sigma'})_{\operatorname{red}}=\overline{\Gamma}_f\cap\left(U_\sigma\times_{U_{\sigma_0}}U_{\sigma'}\right)=\Spec\quotient{\C[\sfS_{\operatorname{split}}]}{\,\vphantom{\sum}I}.\]
Note that $I$ is equal to the toric ideal $I_{\sA\cup\sA'}$ if $\sigma_0=\sigma+\sigma'$.

\end{remark}
%According to Proposition \ref{gfbir_prop} and Theorem \ref{fpcri_prop}, we have
%\[\Spec\quotient{\C[\sfS_{\operatorname{split}}]}{\vphantom{\sum}\sqrt{I_\sA+I_{\sA'}}}=(U_\sigma\times_{U_{\sigma_0}}U_{\sigma'})_{\operatorname{red}}=\overline{\Gamma}_f\cap\left(U_\sigma\times_{U_{\sigma_0}}U_{\sigma'}\right)=\Spec\quotient{\C[\sfS_{\operatorname{split}}]}{\,\vphantom{\sum}I},\]
 %where $\sA$, $\sA'\subseteq M$ generate $\sfS_\sigma$ and $\sfS_{\sigma'}$, respectively, and $I$ is an ideal of $\C[\sfS_{\operatorname{split}}]$ defined in \eqref{gfbir_ideal}. Hence, we conclude that
%\begin{equation}\label{gfbir_rmk_eqn}
%    \sqrt{I_\sA+I_{\sA'}}=I.
%\end{equation}

%%%%%%%%%%%%%%%%%%%%
%%%%%%%%%%%%%%%%%%%%
\subsection{Reduced property}

%\subsection{3-fold}
%Let $f$ be a $3$-dimensional toric flip satisfied the condition \eqref{gftoric_cor3D} in Corollary \ref{gftoric_cor}. 
In this subsection, we study the reduced property of $X$. When we concentrate on a $3$-dimensional toric flip $f$ satisfied the condition \eqref{gftoric_cor3D} in Corollary \ref{gftoric_cor}, by Corollary \ref{gftoric_cor}, Proposition \ref{fpcri_prop}, and Theorem \ref{fpred_thm}, the following statements are equivalent:
\begin{enumerate}[(i)]
    \item $X=X_{\Sigma}\times_{X_{\Sigma_0}}X_{\Sigma'}$ is a normal toric variety.%toric variety,

    \item $X=X_{\widetilde{\Sigma}}$,

    \item $X$ is reduced.
\end{enumerate}
Since the property of being reduced is local, we may assume $X_{\Sigma_0}$ is an affine toric variety $U_{\sigma_0}$ defined by $\sigma_0=\Cone(u_{1},u_{2},u_{3},u_{4})$, and $X_\Sigma$, $X_{\Sigma'}$ are defined by the wall relation \eqref{wallrel_eqn} as in Section \ref{fan_subsec}. To simplify notation, let $\sigma_i=\sigma_{\{1,2,3,4\}\setminus\{i\}}$ and $U_{ji}\coloneqq U_{\sigma_j}\times_{U_{\sigma_0}}U_{\sigma_i}$.

We normalize our wall relation such that $b_i\in\Z$ and $\gcd(b_1, b_2, b_3, b_4) = 1$.  %$\gcd(b_i\mid 1\leq i\leq 4)=1$. 
Note that since $\sigma_4$ is a smooth cone, $\{u_1,u_2,u_3\}$ is a $\Z$-basis of $N$. From the wall relation \eqref{wallrel_eqn}, we have that $b_4\mid b_i$ for $i=1$, $2$, $3$, and thus $b_4=1$.

First, we show a numerical criterion for the reduced property on the affine piece $U_{31}$. %one affine piece.

\begin{lm}\label{n=3_reduced_equiv}
Let $\{a\}_b$ denote the remainder of $a$ divided by $b$. If $g=\operatorname{gcd}(b_1,b_2)>0$ and $b_i=-gb_i'$ for $i=1$, $2$, then the following statements are equivalent:
\begin{enumerate}[(i)]
\setcounter{enumi}{3}
    \item\label{n=3_reduced_equiv_1} $U_{31}$ is reduced.

    \item\label{n=3_reduced_equiv_2} For all $0\leq\lambda\leq b_1'b_2'$, there exists a non-negative integer $y\leq\lambda/b_1'$ such that
    \begin{equation}\label{reduced_ineq}
        \{g\lambda\}_{b_3}\geq g\cdot\{\lambda-b_1'y\}_{b_2'}.
    \end{equation}
    %where $\{n\}_m$ denotes the remainder of $n$ divided by $m$.
\end{enumerate}
\end{lm}

\begin{proof} 
First of all, we have to determine the dual cones of $\sigma=\sigma_3$ and $\sigma'=\sigma_1$, as illustrated in Figure \ref{fig:enter-label}, where we identify $M\simeq\Z^3$ by the dual basis $\{u_1^\vee,u_2^\vee,u_3^\vee\}$ of $\{u_1,u_2,u_3\}$. Note that the coordinates in Figure \ref{fig:enter-label} only represent points on that ray, not that these six points lie in the same plane. 
\begin{figure}%[H]
    \centering
    \resizebox{0.75\textwidth}{!}{%
    \begin{tikzpicture}[line cap=round,line join=round,>=triangle 45,x=1.0cm,y=1.0cm]
    \definecolor{yqyqyq}{rgb}{0.5019607843137255,0.5019607843137255,0.5019607843137255}
    \definecolor{qqqqff}{rgb}{0.,0.,1.}
    \clip(-7.,-2.6) rectangle (8.,5.5);
    \draw [line width=1.5pt,dash pattern=on 2pt off 2pt,color=yqyqyq] (-6.512344383008625,0.5991056435198019)-- (-4.,-2.);
    \draw [line width=1.5pt,dash pattern=on 2pt off 2pt,color=yqyqyq] (-6.512344383008625,0.5991056435198019)-- (6.,-2.);
    \draw [line width=1.5pt,dash pattern=on 2pt off 2pt,color=yqyqyq] (-6.512344383008625,0.5991056435198019)-- (-0.4487836999857161,1.1414858999000128);
    \draw [line width=1.5pt,dash pattern=on 2pt off 2pt,color=yqyqyq] (-6.512344383008625,0.5991056435198019)-- (-2.2728315359382667,2.0300597494773776);
    \draw [line width=1.5pt,dash pattern=on 2pt off 2pt,color=yqyqyq] (-6.512344383008625,0.5991056435198019)-- (-1.,5.);
    \draw [line width=2pt] (-1.,5.)-- (-4.,-2.);
    \draw [line width=2pt] (-4.,-2.)-- (6.,-2.);
    \draw [line width=2pt] (6.,-2.)-- (-1.,5.);
    \draw [line width=2pt] (-1.,5.)-- (0.,-2.);
    \draw [line width=2pt] (-2.2728315359382667,2.0300597494773776)-- (6.,-2.);
    \draw (-0.9,5.5) node[anchor=north west] {$(0,0,-1)$};
    \draw (-3.85,2.6329709125008667) node[anchor=north west] {$(1,0,0)$};
    \draw (-4.437801808647936,-1.990682798982761) node[anchor=north west] {$(b,0,a_0)$};
    \draw (-0.4921031868246658,-1.990682798982761) node[anchor=north west] {$(0,b,a_1)$};
    \draw (5.324751482461186,-1.9364463918099324) node[anchor=north west] {$(-a_1,a_0,0)$};
    \draw (-0.36,1.7) node[anchor=north west] {$(0,1,0)$};
    \draw (-2.05,-0.45) node[anchor=north west] {$\sA_0$};
    \draw (-2.2,2.6) node[anchor=north west] {$\sA\setminus\sA_0$};
    \draw (1.026516214014531,-0.45) node[anchor=north west] {$\sA'\setminus\sA_0$};
    \begin{scriptsize}
    \draw [fill=qqqqff] (-1.,5.) circle (2.0pt);
    \draw [fill=qqqqff] (-4.,-2.) circle (2.0pt);
    \draw [fill=qqqqff] (6.,-2.) circle (2.0pt);
    \draw [fill=qqqqff] (0.,-2.) circle (2.0pt);
    \draw [fill=qqqqff] (-2.2728315359382667,2.0300597494773776) circle (2.0pt);
    \draw [fill=qqqqff] (-0.4487836999857161,1.1414858999000128) circle (2.0pt);
    \draw [fill=qqqqff] (-6.512344383008625,0.5991056435198019) circle (2.0pt);
    \draw[color=black] (-6.593698993767866,0.8228308231077164) node {$O$};
    \end{scriptsize}
    \end{tikzpicture}}
    \caption{Dual cones of $\sigma$ and $\sigma'$.}
    \label{fig:enter-label}
\end{figure}
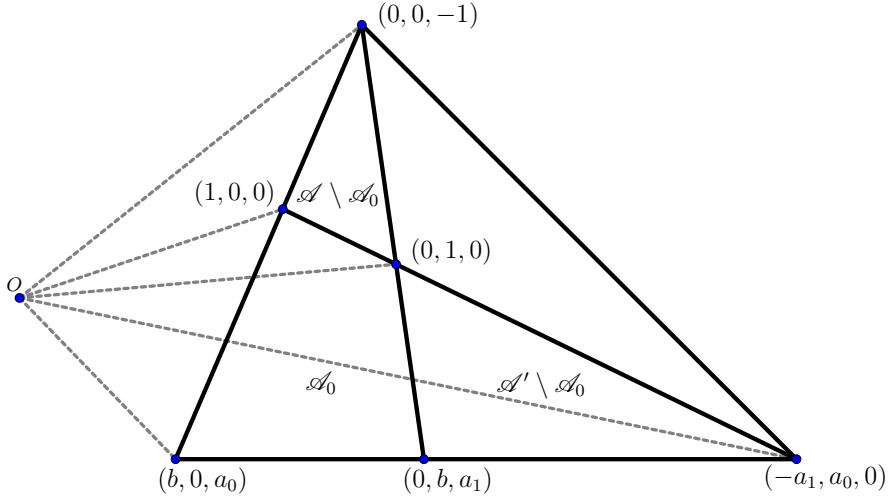

We use the same notation in the proof of Theorem~\ref{fpred_thm}. Observe that (iv) : $U_{31}$ is reduced if and only if (iv') : $x^{\alpha_0+\alpha+\alpha'}-x^{\beta_0+\beta+\beta'}\in I_\sA+I_{\sA'}$ when $(\alpha_0+\alpha+\alpha')-(\beta_0+\beta+\beta')\in L_\sB$. 

The {\bf  key point} is that the generating set $\sA$ can be selected to be $\sA_0\cup\{-u_3^\vee\}$. Let $\gamma\in\Z_{\geq0}^{\sA\setminus\sA_0}$ such that $-u_3^\vee=m_\gamma$. Notice that $\gamma$ is the unique vector in $\eZ^{\sA}$ such that $m_\gamma=-u_3^\vee$.

Define two semi-groups
\[\Gamma=M\cap\overline{\sigma'^\vee\setminus\sigma^\vee},\quad \Gamma'=\Gamma\cap\Cone(u_3)^\perp.\]
We {\bf claim} that (iv') holds if and only if (iv'') : $\Gamma\setminus\Gamma'\subseteq(\sfS_{\sigma_0}\setminus 0)+\Gamma'$.
Assume that (iv') holds. For any $m\in\Gamma\setminus\Gamma'$, it is clear that $m-u_3^\vee\in\sfS_{\sigma'}$ since
\[\gen{m-u_3^\vee,u_3}=\gen{m,u_3}-1\geq 0.\]
Take a lifting $\alpha_0,\beta_0\in\Z_{\geq 0}^{\sA_0}$, 
$\alpha',\beta'\in\Z_{\geq 0}^{\sA'\setminus\sA_0}$ such that $m=m_{\alpha_0+\alpha'}$ and $m-u_3^\vee=m_{\beta_0+\beta'}$. Then $x^{\alpha_0+\gamma+\alpha'}-x^{\beta_0+\beta'}\in I_\sA+I_{\sA'}$ since 
$(\alpha_0+\gamma+\alpha')-(\beta_0+\beta')\in L_\sB$. By lemma \ref{gfbir_lm}, there exists a sequence %$\{(\gamma_{0i},\gamma_i,\gamma_i')\}_{i=1}^m\subseteq\eZ^{\sA_0}\oplus\eZ^{\sA\setminus\sA_0}\oplus\eZ^{\sA'\setminus\sA_0}$
$\{(\gamma_{0i},\gamma_i,\gamma_i')\}_{i=1}^m\subseteq\sfSp$
such that
\[(\gamma_{01},\gamma_1,\gamma_1')=(\alpha_0,\gamma,\alpha')\quad\text{and}\quad(\gamma_{0m},\gamma_m,\gamma_m')=(\beta_0,0,\beta'),\]
and for each $1\leq i\leq m-1$, either \eqref{relation1} or \eqref{relation2} holds. If $i$ is the smallest index such that $\gamma_{i+1}\neq \gamma$, then \eqref{relation2} must holds for $i$, that is,
\[(\gamma_{0i}+\gamma_i)-(\gamma_{0(i+1)}+\gamma_{i+1})\in L_\sA.\]
If $m_{\gamma_{0i}}=0$, then $m_{\gamma_{0(i+1)}}+m_{\gamma_{i+1}}=m_{\gamma_{i}}=-u_3^\vee$ in the face $\sigma_{\{12\}}^*$ of $\sigma^\vee$. This implies that $m_{\gamma_{0(i+1)}}\in\sigma_0^\vee\cap\sigma_{12}^*=\{0\}$ and $m_{\gamma_{i+1}}=-u_3^\vee$, and thus $\gamma_{i+1}=\gamma$, a contradiction. Hence $m_{\gamma_{\sigma_i}}\neq 0$. Since $m_{\gamma_{0i}}+m_{\gamma_i}+m_{\gamma_i'}$ is a constant, we have
\[m_{\alpha_0}+m_\gamma+m_{\alpha'}=m_{\gamma_{0i}}+m_{\gamma_i}+m_{\gamma_i'}\implies m=m_{\alpha_0+\alpha'}=m_{\gamma_{0i}}+m_{\gamma_i'}\in (\sfS_{\sigma_0}\setminus 0)+\Gamma.\]

% If $U_{00}$ is reduced, then $YZ-Z'\in I+J$. Note that an element in $I+J$ containing the term $Y$ is 
% of the form 
% \[\sum_i c_i(Y^{n_i}X_{1,i}-X_{2,i})Y^{m_i}Z_i+\sum_j d_j Y^{n_j}(X_{3,j}Z_{3,j}-X_{4,j}Z_{4,j}),\]
% where $n_iy+x_{1,i}=x_{2,i}$, $x_{3,j}+z_{3,j}=x_{4,j}+y_{4,j}$, and $c_i$, $d_j\in\C$. Since $YZ-Z'$ is not divided by $Y$, we conclude that $z\in (\Z_{\geq 0}\mathcal{X}\setminus 0)+\Z_{\geq 0}\mathcal{Z}$. 

Furthermore, let $m_1=m\in \Gamma\setminus\Gamma'$. Since $\Gamma\setminus\Gamma'\subseteq(\sfS_{\sigma_0}\setminus 0)+\Gamma$, there exists $m_2\in\Gamma$ such that $m_1-m_2\in(\sfS_{\sigma_0}\setminus 0)$. By induction, if $m_i\in\Gamma\setminus\Gamma'$, then there exists $m_{i+1}\in\Gamma$ such that $m_i-m_{i+1}\in(\sfS_\sigma\setminus 0)$. This process can be continuously done until $m_{i+1}$ no longer lies in $\Gamma\setminus\Gamma'$. Note that this process will terminate, since $(\sigma')^\vee$ is strongly convex.
% write $m_1=m_{2,0}+m_2$ where $m_{2,0}\in S_{\sigma_0}$ and $m_2\in\overline{\Gamma}$. If $m_2\in\Gamma$, then we can again write $m_2=m_{3,0}+m_{3}$ where $m_{3,0}\in S_{\sigma_0}$ and $m_{3}\in\overline{\Gamma}$. 
% Otherwise, since the 2nd and 3rd coordinates of a lattice point in $\Z_{\geq 0}(\mathcal{X}\cup\mathcal{Z})$ are non-negative integers, the 2nd and 3rd coordinates of $\{z_i\}_{i\geq 0}$ form two sequences of decreasing non-negative integers and thus will be stable. Similarly, the 1st coordinates of $\{z_i\}_{i\geq 0}$ forms a sequence of decreasing integer. Pick $n\in\N$ such that $\{z_i\}_{i\geq n}$ has the same 2nd and 3rd coordinates, which implies that $z_n+\Z e_1^\vee\subseteq\Gamma$ and contradicts the fact that $\{z_i\}_{i\geq 0}$ is an infinite sequence. 
Hence (iv'') holds. Conversely, we assume that (iv'') holds. It is clear that (iv'') is equivalent to
\begin{equation}\label{Gamma in Gamma'}
\Gamma\subseteq\sfS_{\sigma_0}+\Gamma'.
\end{equation}
For any $\alpha'\in\Z_{\geq 0}^{\sA'\setminus\sA_0}$ such that $m_{\alpha'}\in\Gamma$, we can decompose $m_{\alpha
'}=m_{\beta_0}+m_{\beta'}$ by \eqref{Gamma in Gamma'} and get
\begin{equation}\label{inI+I_eq1}
    x^{\alpha'}-x^{\beta_0+\beta'}\in I_{\sA'},
\end{equation}
where $m_{\beta_0}\in\sfS_{\sigma_0}$, $m_{\beta'}\in\Gamma'$ for some $\beta_0\in\eZ^{\sA_0}$ and $\beta'\in\eZ^{\sA'}$. Since
\[0<\gen{m_{\alpha'},u_3}=\gen{m_{\beta_0},u_3},\]
we still have $m_{\beta_0}-u_3^\vee\in\sfS_{\sigma_0}$ and $m_{\alpha'}-u_3^\vee\in\sfS_{\sigma'}$. Take a lifting $\widetilde{\beta}_0\in\Z_{\geq 0}^{\sA_0}$, $\alpha''\in\Z_{\geq 0}^{\sA'}$ such that $m_{\widetilde{\beta}_0}=m_{\beta_0}-u_3^\vee$ and $m_{\alpha''}=m_{\alpha'}-u_3^\vee$. Then $m_{\alpha''}=m_{\widetilde{\beta}_0}+m_{\beta'}$, and thus
\begin{equation}\label{inI+I_eq2}
    x^{\beta_0+\gamma}-x^{\widetilde{\beta}_0}\in I_\sA, \quad x^{\alpha''}-x^{\widetilde{\beta}_0+\beta'}\in I_{\sA'}.
\end{equation}
From \eqref{inI+I_eq1} and \eqref{inI+I_eq2}, we conclude
\begin{equation}\label{in I+I}
x^{\gamma+\alpha'}-x^{\alpha''}\in I_\sA+I_{\sA'}.
\end{equation}
Now, suppose that (iv') is false, say, there exists a relation 
\[(\alpha_0+k_1\gamma+\alpha')-(\beta_0+k_2\gamma+\beta')\in L_\sB\]
such that $x^{\alpha_0+k_1\gamma+\alpha'}-x^{\beta_0+k_2\gamma+\beta'}\notin I_\sA+I_{\sA'}$ for some $\alpha_0$, $\beta_0\in\Z_{\geq 0}^{\sA_0}$, $\alpha'$, $\beta'\in\Z_{\geq 0}^{\sA'}$ and $k_1$, $k_2\in\Z_{\geq 0}$. We may assume that $k_2=0$ by eliminating $\min\{k_1,k_2\}y$, and $k_1\geq 0$ is the smallest integer such that the above relation holds (note that $k_1\neq 0$). We find that $m_{\alpha'}\in\Gamma'$, otherwise, let $\alpha''\in\Z_{\geq 0}^{\sA'}$ such that $m_{\alpha''}=m_{\alpha'}+m_\gamma\in\Gamma\setminus\Gamma'$. Then the binomial corresponds to the new relation
\[(\alpha_0+(k_1-1)\gamma+\alpha'')-(\beta_0+\beta')\in L_\sB\]
is not in $I_\sA+I_{\sA'}$ by \eqref{in I+I}, which leads to a contradiction of the minimality of $k_1$. Note that $m_{\alpha_0+k_1\gamma}\in S_{\sigma_0}$ since $m_{\alpha'}\in\overline{\Gamma}$ implies
\[\gen{m_{\alpha_0+k_1\gamma},u_3}=\gen{m_{\alpha_0+k_1\gamma}+m_{\alpha'},u_3}=\gen{m_{\beta_0+\beta'},u_3}\geq 0.\]
Then there exists $\widetilde{\alpha}_0\in\Z_{\geq 0}^{\sA_0}$ such that $m_{\widetilde{\alpha}_0}=m_{\alpha_0+k_1\gamma}$. Then we get the contradiction from
\[x^{\alpha_0+k_1\gamma}-x^{\widetilde{\alpha}_0}\in I_\sA,\quad x^{\widetilde{\alpha}_0+\alpha'}-x^{\beta_0+\beta'}\in I_{\sA'}\implies x^{\alpha_0+k_1\gamma+\alpha'}-x^{\beta_0+\beta'}\in I_\sA+I_{\sA'}\]
and thus (iv') is true. Now we have finished the {\bf proof of our claim}.

Finally, it suffices to show that \eqref{Gamma in Gamma'} is equivalent to the condition \eqref{n=3_reduced_equiv_2}. Given $m=(-p_1,p_2,p_3)\in\Gamma$, that is, $p_1$, $p_3\geq 0$ and $-b_2p_2\geq -b_1p_1+b_3p_3$, we want to find a lattice point $m'=(-q_1,q_2,0)\in\Gamma'$ such that $m-m'=(q_1-p_1,p_2-q_2,p_3)\in S_{\sigma_0}$, that is, there exist $q_1$, $q_2\geq 0$ such that
\begin{equation}\label{v.1}
-b_2q_2\geq -b_1q_1,\quad q_1\geq p_1,\quad q_2\leq p_2,\quad k\coloneqq b_1p_1-b_2p_2-b_3p_3\geq -b_2q_2+b_1q_1.  
\end{equation}
Note that if $(q_1,q_2)$ satisfies \eqref{v.1}, then $(q_1,\lceil\frac{b_1q_1}{b_2}\rceil)$ satisfies \eqref{v.1}. Thus, \eqref{v.1} is equivalent to the existence of an integer $q_1\in[p_1,b_2p_2/b_1]$ such that
\[k\geq -b_2\left\lceil\dfrac{b_1q_1}{b_2}\right\rceil+b_1q_1=g\cdot \{-b_1'q_1\}_{b_2'}.\]
Consider the following statement $P(p_1,p_2):$ 
\[\text{there exists } p_1\leq q_1\leq\dfrac{b_2'p_2}{b_1'} \text{ such that }\{g(b_2'p_2-b_1'p_1)\}_{b_3}\geq g\cdot\{-b_1'q_1\}_{b_2'}.\]
Hence, \eqref{Gamma in Gamma'} is also equivalent to $P(p_1,p_2)$ holds for all $p_1,p_2>0$ with $b_2'p_2\geq b_1'p_1$. Note that if $b_2'p_2/b_1'-p_1\geq b_2'$, then there exists an integer $q_1\in[p_1,b_2'p_2/b_1']$ such that $b_2'|q_1$, ensuring that the above inequality holds. Therefore, we only need to check it for $\lambda\coloneqq b_2'p_2-b_1'p_1\leq b_1'b_2'$. We take $n_1,n_2\in\N$ such that $b_2'n_2-b_1'n_1=1$. Since $P(p_1,p_2)$ holds if and only if $P(p_1+b_2',p_2+b_1')$ holds, we only need to check that $(p_1,p_2)=(\lambda n_1,\lambda n_2)$ for $0\leq \lambda\leq b_1'b_2'$. Let $q_1=\lambda n_1+y$. Clearly, $P(\lambda n_1,\lambda n_2)$ can be reformulated as the existence of
\[0\leq y\leq \dfrac{b_2'\cdot \lambda n_2}{b_1'}-\lambda n_1=\dfrac{\lambda}{b_1'}\]
such that
\[\{g\lambda\}_b\geq g\cdot\{-b_1'(\lambda n_0+y)\}_{b_2'}=g\cdot \{\lambda-b_1'y\}_{b_2'}.\]
This establishes the equivalence \eqref{n=3_reduced_equiv_1}$\Leftrightarrow$\eqref{n=3_reduced_equiv_2}.
\end{proof}

%in the above proof
% For (v)$\implies$(vi), simply by taking $\lambda=b$.

% For (vi)$\implies$(v), let $\lambda=pb+q(a_0'+a_1')+r$ where 
% \[0\leq q\leq g-1 \text{ and } 0\leq r\leq a_0'+a_1'-1-\delta_{q,g-1}.\]
% We take $y=py_0+q\leq pb/a_0'+q\leq\lambda/a_0'$ and get that
% \[\{g\lambda\}_b=\{gr+q\}_b=gr+q,\]
% since $gr+q\leq g(a_0'+a_1'-1)+g-1=b$ and the equality will not hold. On the other hand,
% \[\{\lambda-a_0'y\}_{a_1'}=\{p(b-a_0'y_0)+qa_0'+r-qa_0'\}_{a_1'}=\{r\}_{a_1'}.\]
% Hence
% \[\{g\lambda\}_b=gr+b\geq g\{r\}_{a_1'}=\{\lambda-a_0'y\}_{a_1'}.\]

\begin{rmk}\label{U31 iff U32}
Notice that the condition \eqref{n=3_reduced_equiv_2} is symmetric with respect to $b_1$ and $b_2$. Indeed, given $0\leq\lambda\leq b_1'b_2'$, let $0\leq y_1\leq\lambda/b_1'$ satisfy the inequality \eqref{reduced_ineq}, that is, there exists $0\leq y_2\leq\lambda/b_2'$ such that
\[\lambda-y_1b_1'-y_2b_2'=\{\lambda-y_1b_1'\}_{b_2'}.\]
Then the inequality
\[\{\lambda-y_2b_2'\}_{b_1'}=\{\{\lambda-y_1b_1'\}_{b_2'}\}_{b_1'}\leq\{\lambda-y_1b_1'\}_{b_2'}\leq\{g\lambda\}_{b_3}\]
as required. Hence we conclude that $U_{31}$ is reduced if and only if $U_{32}$ is reduced.
\end{rmk}

In fact, we can achieve a stronger result. %The main theorem of this section is as follows.
\begin{thm}\label{r=1 reduced}
$X$ is reduced if and only if $U_{31}$ is reduced.
\end{thm}

This result is based on the following lemma.

\begin{lm}\label{split}
Let $\pi\colon X = X_{\Sigma} \times_{X_{\Sigma_0}} X_{\Sigma'} \to X_{\Sigma}$ be the first projection and $\sI$ be the nilradical ideal of $\O_X$. Then the following vanishing results hold for the 3-dimensional case : 
\[\pi_*\mathscr{I}=R^i\pi_*\mathscr{I}=R^i\pi_*\O_X=0\] 
for all $i>0$, and $\pi_*\O_X=\O_{X_{\Sigma}}$. 

Also, the similar results hold for the second projection $\pi':X\to X_{\Sigma'}$.
\end{lm}

\begin{proof}
Since $\pi$ is a proper morphism with fiber dimension $\leq 1$, by the formal function theorem, $R^i\pi_*\mathscr{I}=R^i\pi_*\O_X=0$ for all $i>1$. For $i=1$, it suffices to show that
\[H^1(\pi^{-1}(U_{\sigma_j}),\mathscr{I})=H^1(\pi^{-1}(U_{\sigma_j}),\O_X)=0, \text{ for } j=3,4.\]
Since $\pi^{-1}(U_{\sigma_j})$ is covered by $\left\{U_{jk}\coloneqq U_{\sigma_j}\times_{U_\sigma} U_{\sigma_k}\right\}_{k=1,2}$, by \v{C}ech cohomology, it suffices to show that
\begin{align}\label{Cech1}
H^0(U_{j1},\mathscr{I})\oplus H^0(U_{j2},\mathscr{I}) & \longrightarrow H^0(U_{j1}\cap U_{j2},\mathscr{I})=H^0(U_{\sigma_j}\times_{U_{\sigma_0}}U_{\sigma_1\cap\sigma_2},\mathscr{I})\\
\label{Cech2}
H^0(U_{j1},\O_X)\oplus H^0(U_{j2},\O_X) & \longrightarrow H^0(U_{j1}\cap U_{j2},\O_X)=H^0(U_{\sigma_j}\times_{U_{\sigma_0}}U_{\sigma_1\cap\sigma_2},\O_X) 
\end{align}
are surjective. 

By the relation $b_1u_1+b_2u_2+b_3u_3+b_4u_4=0$, we have $(\sigma_1\cap\sigma_2)^\vee= \sigma_1^{\vee}\cup\sigma_2^{\vee}$, and thus
\begin{equation}\label{split cone sequence}
\begin{matrix}
0\to\C[\sigma_1^{\vee}\cap\sigma_2^{\vee}\cap M] & \longrightarrow & \C[\sigma_1^{\vee}\cap M]\oplus \C[\sigma_2^{\vee}\cap M] & \longrightarrow & \C[(\sigma_1\cap\sigma_2)^\vee\cap M]\to 0 \\
& & (\chi^{m_1},\chi^{m_2}) & \longmapsto &\chi^{m_1}-\chi^{m_2}
\end{matrix}
\end{equation}
Since \eqref{Cech2} is equal to the last morphism in \eqref{split cone sequence} tensored by $\C[\sigma_j^\vee\cap M]$ over $\C[\sigma_0^\vee\cap M]$, we conclude that \eqref{Cech2} is surjective.

We modify the notation in Theorem \ref{gfbir_prop} as follows:
\[\sfS_{\sigma_0}=\Z_{\geq 0}\sA_0,\quad \sfS_{\sigma_j}=\Z_{\geq 0}\sA,\quad \sfS_{\sigma_k}=\Z_{\geq 0}\sA'_k\quad \text{ for } k=1,2.\]
Since $\sfS_{\sigma_1\cap\sigma_2}=\sfS_{\sigma_1}+\sfS_{\sigma_2}=\Z_{\geq 0}(\sA_1'\cup\sA_2')$ and $\sigma_0=\sigma_j+(\sigma_1\cap\sigma_2)$, according to Remark \ref{gfbir_rmk} we get the equality
\[\sqrt{I_{\sA}+I_{\sA_1'\cup\sA_2'}}=I_{\sA\cup(\sA_1'\cup\sA_2')}\]
in $\C[\Z_{\geq 0}^{\sA\cup\sA_1'\cup\sA_2'}]$, and thus
\[H^0(U_{\sigma_j}\times_{U_{\sigma_0}}U_{\sigma_1\cap \sigma_2},\mathscr{I})=\quotient{I_{\sA\cup\sA_1'\cup\sA_2'}}{I_\sA+I_{\sA'_1+\sA'_2}}\]
is generated by $\overline{x^{\alpha+\alpha'}-x^{\beta+\beta'}}$, where $\alpha,\beta\in\Z_{\geq 0}^{\sA\setminus\sA_0}$, $\alpha',\beta'\in\Z_{\geq 0}^{\sA_1'\cup\sA_2'}$ with $(\alpha+\alpha')-(\beta+\beta')\in L_{\sA\cup\sA'_1\cup\sA'_2}$.

If $m_{\alpha'}$, $m_{\beta'}\in S_{\sigma_k}$ for some $k\in\{1,2\}$, say $\alpha''$, $\beta''\in\Z_{\geq 0}^{\sA_k'}$ such that $m_{\alpha'}=m_{\alpha''}$ and $m_{\beta'}=m_{\beta''}$, then
\[x^{\alpha+\alpha'}-x^{\beta+\beta'}\equiv x^{\alpha+\alpha''}-x^{\beta+\beta''} \pmod{I_{\sA_1'+\sA_2'}},\]
and $(-1)^{k-1}\overline{x^{\alpha+\alpha''}-x^{\beta+\beta''}}\in I_{\sA\cup\sA_k'}/(I_{\sA}+I_{\sA_k'})$ maps to $\overline{x^{\alpha+\alpha'}-x^{\beta+\beta'}}\in I_{\sA\cup\sA_1'\cup\sA_2'}/(I_{\sA}+I_{\sA_1'+\sA_2'})$.

If not, we may assume that $m_{\alpha'}\in \sfS_{\sigma_1}$ and $m_{\beta'}\in \sfS_{\sigma_2}$. Since 
\[m_{\alpha}+m_{\alpha'}=m_\beta+m_{\beta'}\in (\sfS_{\sigma_j}+\sfS_{\sigma_1})\cap(\sfS_{\sigma_j}+\sfS_{\sigma_2})\subseteq \sfS_{\sigma_j\cap\sigma_1}\cap \sfS_{\sigma_j\cap \sigma_2}\subseteq \sfS_{\sigma_j},\]
there exist $\gamma_0\in\Z_{\geq 0}^{\sA_0}$ and $\gamma\in\Z_{\geq 0}^{\sA\setminus\sA_0}$ such that
\[m_{\gamma_0}+m_{\gamma}=m_{\alpha}+m_{\alpha'}=m_\beta+m_{\beta'}.\]
Then
\[(\overline{x^{\alpha+\alpha'}-x^{\gamma_0+\gamma}},\overline{x^{\beta+\beta'}-x^{\gamma_0+\gamma}})\mapsto\overline{x^{\alpha+\alpha'}-x^{\beta+\beta'}},\]
under the morphism $H^0(U_{j1},\O_X)\oplus H^0(U_{j2},\O_X)\to H^0(U_{j1}\cap U_{j2},\O_X)$. Hence, we conclude that \eqref{Cech1} is surjective.

Note that
\begin{equation}\label{pushforward of nilpotent}
0\to\pi_*\mathscr{I}\to\pi_*\O_X\to\pi_*\O_{X_{\operatorname{red}}}\to R^1\pi_*\mathscr{I}=0.
\end{equation}
By Corollary \ref{gftoric_cor} and Theorem \ref{fpred_thm}, $X_{\operatorname{red}}=X_{\widetilde{\Sigma}}$ is integral. By the proof of Zariski main theorem in \cite{Hartshorne}, $\pi_*\O_{X_{\operatorname{red}}}=\O_{X_{\Sigma}}$. By functoriality, the morphism between structure sheaves $\O_{X_{\Sigma}}\to\pi_*\O_X$ gives a lifting of \eqref{pushforward of nilpotent}, and thus the short exact sequence \eqref{pushforward of nilpotent} splits.

Now the remaining part is to prove that $\pi_*\mathscr{I}=0$. By definition and $R^1\pi_*\O_X=0$, we have
\[0\to\O_X(\pi^{-1}U_{\sigma_j})\to\O_X(U_{j1})\oplus\O_X(U_{j2})\to\O_X(U_{\sigma_j}\times_{U_{\sigma_0}}U_{\sigma_1\cap\sigma_2})\to 0,\]
or rewrite it as 
\[
0\xrightarrow{\phantom{\alpha}} \C[\sfS_{\sigma_j}]\oplus\Gamma(\pi_*\mathscr{I},U_{\sigma_j}) \xrightarrow{\phantom{\alpha}} \C[\sfS_{\sigma_j}]\otimes_{\C[\sfS_{\sigma_0}]}\left(\C[\sfS_{\sigma_1}]\oplus \C[\sfS_{\sigma_2}]\right) 
\xrightarrow{\alpha} \C[\sfS_{\sigma_j}]\otimes_{\C[\sfS_{\sigma_0}]}\C[\sfS_{\sigma_1\cap\sigma_2}]\to 0, 
\]
since $\pi_*\O_X=\pi_*\mathscr{I}\oplus\O_{X_{\Sigma}}$. From the long exact sequence induced by $\eqref{split cone sequence}\otimes_{\C[\sfS_{\sigma_0}]}\C[\sfS_{\sigma_j}]$, the kernel of $\alpha$ is 
\[\quotient{\C[\sfS_{\sigma_j}]}{\ \operatorname{Im}\left(\Tor_1^{\C[\sfS_{\sigma_0}]}\left(\C[\sfS_{\sigma_1\cap\sigma_2}],\C[\sfS_{\sigma_j}]\right)\to \C[\sfS_{\sigma_j}]\right)}.\]
So we can conclude that $\Gamma(\pi_*\mathscr{I},U_{\sigma_j})=0$. Hence $\pi_*\mathscr{I}|_{U_{\sigma_j}}=0$ for $j=1$, $2$, that is, $\pi_*\mathscr{I}=0$.
\end{proof}

\begin{proof}[Proof of Theorem~\ref{r=1 reduced}]
Suppose that $U_{31}$ is reduced. 
Recall that $J_{-} = \{1, 2\}$ and $J_{+} = \{3, 4\}$. Let $Z = V(\sigma_{J_{-}})$, $Z' = V(\sigma_{J_{+}})$ and $S = V(\sigma_{J_{-} \cup J_{+}})$. By Remark~\ref{U31 iff U32}, $U_{32}$ is also reduced, and thus
\[\operatorname{Supp}\mathscr{I}\subseteq(Z\times_S Z')_{\operatorname{red}}\setminus\left(U_{31}\cup U_{32}\right)=(\pi')^{-1}(p),\]
where $p$ is the unique point in $V(\sigma_1\cap\sigma_2)\setminus U_{\sigma_1}=V(\sigma_2)$. Note that $(Z\times_SZ')_{\operatorname{red}}=\P^1\times\P^1$, and $\pi|_{\P^1\times\P^1}$ and $\pi'|_{\P^1\times\P^1}$ are projection onto each component. We take a section $s\colon Z=\P^1\xrightarrow{\sim} (\pi')^{-1}(p)_{\operatorname{red}}=\P^1$ such that $\pi\circ s=\operatorname{id}_Z$. Let $\mathscr{F}=(s^{-1})_*\iota'^{-1}\mathscr{I}$ be the sheaf of abelian groups on $Z$, where $\iota' \colon (\pi')^{-1}(p)_{\operatorname{red}}\inclusion X$. Since $\mathscr{I}$ has the support on $(\pi')^{-1}(p)_{\operatorname{red}}$ and by Lemma~\ref{split}, we have
\[\iota_*\mathscr{F}=\pi_*\iota'_*s_*(s^{-1})_*\iota'^{-1}\mathscr{I}=\pi_*\mathscr{I}=0,\]
where $\iota\colon Z\inclusion X_{\Sigma}$. This implies $\mathscr{F}=0$, and thus $\mathscr{I}=0$, i.e., $X$ is reduced. 
\end{proof}

%%%%%%%%%%%%%%%%%%%%%%%%%%%%%%%%%%%%%%%%
%%%%%%%%%%%%%%%%%%%%%%%%%%%%%%%%%%%%%%%%
\section{Applications}

%%%%%%%%%%%%%%%%%%%%
%%%%%%%%%%%%%%%%%%%%
\subsection{About terminal and canonical singularities}
%\subsection{Terminal 3-fold}

Let $\phi \colon X_{\Sigma} \to X_{\Sigma_0}$ be the flipping contraction of an extremal ray $\cR$ and $\phi' \colon X_{\Sigma'} \to X_{\Sigma_0}$ be its corresponding flip. We can apply the results in Section \ref{fp_sec} to study $X = X_{\Sigma}\times_{X_{\Sigma_0}}X_{\Sigma'}$ when $X_{\Sigma}$ is a $3$-dimensional simplicial toric variety with at worst terminal or canonical singularities.

\begin{theorem}\label{term3flip_thm}
Let $X_{\Sigma}$ be a $3$-dimensional simplicial toric variety with at worst terminal singularities. If  $K_{X_\Sigma} \cdot \cR\leq 0$, then the fiber product $X = X_{\Sigma}\times_{X_{\Sigma_0}}X_{\Sigma'}$ is the toric variety $X_{\tSigma}$ defined by the fan \eqref{fp_fan}.% a normal toric variety.
\end{theorem}

%\begin{theorem}
%Let $\phi \colon X_{\Sigma} \to X_{\Sigma_0}$ be the flipping contraction of an extremal ray $\cR$ with $K_{X_\Sigma} \cdot \cR < 0$ from a $3$-dimensional simplicial toric variety with at worst terminal singularities, and let $\phi' \colon X_{\Sigma'} \to X_{\Sigma_0}$ be the flip of it. Then the fiber product $X = X_{\Sigma}\times_{X_{\Sigma_0}}X_{\Sigma'}$ is the toric variety $X_{\tSigma}$ defined by the fan \eqref{fp_fan}.% a normal toric variety.
%\end{theorem}

%\begin{thm}
%Consider the flip $X_{\Sigma}\rightarrow X_{\Sigma_0}\leftarrow X_{\Sigma'}$ from a simplicial toric $3$-fold with at worst terminal singularities. Then the fiber product $X=X_{\Sigma}\times_{X_{\Sigma_0}}X_{\Sigma'}$ is a toric variety.
%\end{thm}

\begin{proof}
We have known that $X$ is irreducible by Theorem \ref{fpred_thm}. Without loss of generality, we may assume that $X_{\Sigma_0}$ is affine. By the classification result for $3$-dimensional terminal singularities \cite[Corollary 2.1, Theorem 3.1]{FSTU09} (see also \cite[14-2-5]{Matsuki}), the wall relation for $K_{X_\Sigma}.\cR<0$ is given by
\[
    -au_1-(r-a)u_2+ru_3+u_4=0 \quad\mbox{ or }\quad -au_1-u_2+ru_3+u_4=0,
\]
and for $K_{X_\Sigma}.\cR=0$ is given by
\[-u_1-u_2+u_3+u_4=0,\]
where $\{u_1, u_2, u_3\}$ is a $\bZ$-basis of $N = \bZ^3$, $0 < a< r$ and $\gcd(a, r) = 1$. In particular, the toric variety $X_{\Sigma}$ satisfies Corollary \ref{gftoric_cor} \eqref{gftoric_cor3D} and thus the graph closure $\Graphc$ is normal, where $f = (\phi')^{- 1} \circ \phi$. 
%So we can apply 

To see the reduced property of $X$, by Theorem~\ref{r=1 reduced}, it suffices to show that the numerical criterion \eqref{n=3_reduced_equiv_2} in Lemma~\ref{n=3_reduced_equiv} holds. For the first case, we have $g=1$. Given $\lambda\geq 0$, we can take $y=\lfloor\frac{\lambda}{r}\rfloor\leq\frac{\lambda}{a}$. Then
\[\lambda-ay=ry+\{\lambda\}_r-ay\implies\{\lambda-ay\}_{r-a}=\{\{\lambda\}_r\}_{r-a}\leq\{g\lambda\}_r.\]
For the second and third cases, given any $\lambda\geq 0$ we can take $y=0$ since $b_2'=1$. Therefore $X$ is a normal toric variety by Proposition \ref{fpcri_prop}.
\end{proof}

%\subsection{Canonical 3-fold}

We now assume for simplicity that $X_{\Sigma_0} = U_{\sigma_0}$. Recall that the extremal ray $\cR$ is defined by the wall relation 
\begin{equation}\label{wallrel_eqn2}
    \sum_{i \in J_{-}} b_i u_i + \sum_{j \in J_{+}} b_j u_j = 0.
\end{equation}
After suitable scaling, we may assume that $b_i\in\Z$ for $i\in J_-\cup J_+$ and $\gcd(b_1,\ldots,b_{n+1})=1$.

The following proposition illustrates the fiber product $X$ may be not reduced when $X_\Sigma$ has canonical singularities (see Remark \ref{can3_rmk}).

\begin{proposition}\label{can3flop_prop}
Suppose that $\{u_1, u_2, u_3\}$ is a $\bZ$-basis of $N \simeq \bZ^3$, and $J_{-} = \{1, 2\}$, $J_{+} = \{3, 4\}$ and $b_4 = 1$. Assume further that 
\begin{equation}\label{flop_eqn}
    b_1+b_2+b_3+1=0.
\end{equation}
Then the fiber product $X$ is the toric variety $X_{\tSigma}$ if and only if there exist two non-negative integers $y_1$ and $y_2$ such that 
\begin{equation}\label{b=b+b}
b_3=b_1'y_1+b_2'y_2
\end{equation}
where $g=\operatorname{gcd}(b_1,b_2)>0$ and $b_i=-gb_i'$ for $i=1$, $2$.
\end{proposition}

%In this subsection, we focus on the sub-case of $K$-flop with wall relation
%\[b_1u_1+b_2u_2+b_3u_3+u_4=0,\]
%where $\{u_1,u_2,u_3\}$ forms a basis of $N=\Z^3$. 

%\begin{thm}
%The fiber product $X$ is a toric variety if and only if there exists $y_1$, $y_2\in\Z_{\geq 0}$ such that 
%\begin{equation}\label{b=b+b}
%b_3=b_1'y_1+b_2'y_2,
%\end{equation}
%where $g=\operatorname{gcd}(b_1,b_2)>0$ and $b_i=-gb_i'$ for $i=1$, $2$.
%\end{thm}

\begin{proof}
The irreducibility of $X$ and the normality of $\Graphc$ follow as in the proof of Theorem \ref{term3flip_thm}. By Theorem \ref{r=1 reduced} and Proposition \ref{fpcri_prop}, the proposition follows from the claim that $\eqref{b=b+b}$ is equivalent to \eqref{n=3_reduced_equiv_2} in Lemma \ref{n=3_reduced_equiv}.

For \eqref{n=3_reduced_equiv_2}$\Rightarrow$\eqref{b=b+b}, according to the proof in Lemma \ref{n=3_reduced_equiv}, the statement \eqref{n=3_reduced_equiv_2} in Lemma \ref{n=3_reduced_equiv} holds for all $\lambda\geq 0$. In particular, for $\lambda=b_3$, there exists $0\leq y_1\leq b/b_1'$ such that
\[\{b_3-b_1'y_1\}_{b_2'}\leq \{gb_3\}_{b_3}=0.\]
That is, there exists $y_2\in\Z_{\geq 0}$ with $b_3-b_1'y_1=b_2'y_2$ as required.

For \eqref{b=b+b}$\Rightarrow$\eqref{n=3_reduced_equiv_2}, let $\lambda=pb_3+q(b_1'+b_2')+r$ where 
\[0\leq q\leq g-1\quad \text{and}\quad 0\leq r\leq b_1'+b_2'-1-\delta_{q,g-1}.\]
We take $y=py_1+q\leq pb_3/b_1'+q\leq\lambda/b_1'$, and notice that
\[\{g\lambda\}_{b_3}=\{q(b_1+b_2)+gr\}_{b_3}=\{q+gr\}_{b_3}=q+gr,\]
since $0 \leq q+gr\leq g-1+g(b_1'+b_2'-1 - \delta_{q, g - 1}) < b_3$. % and the equality will not hold. 
On the other hand,
\[\{\lambda-b_1'y\}_{b_2'}=\{p(b_3-b_1'y_1)+qb_1'+r-qb_1'\}_{b_2'}=\{r\}_{b_2'}.\]
Hence
\[\{g\lambda\}_{b_3}=q+gr\geq g\cdot\{r\}_{b_2'}=g\cdot\{\lambda-b_1'y\}_{b_2'}.\]
\end{proof}

\begin{remark}\label{can3_rmk}
%By \cite[Proposition 11.4.12]{CLS},
Note that 
the singularities of $X_{\Sigma}$ in Proposition \ref{can3flop_prop} are at worst canonical singularities. Indeed, since $\sigma_4$ is smooth cone, we only need to check the condition (b) in \cite[Proposition 11.4.12]{CLS} for $\sigma_3$, where $\sigma_i = \sigma_{\{1, 2,3, 4\} \setminus \{i\}}.$ Set the polytope $\Pi_{\sigma_3} = \operatorname{Conv}(0,u_1,u_2,u_4)$. For $m \in \Pi_{\sigma_3} \cap M$, there are $a$, $a_i\in\Z_{\geq 0}$ and $a_1+a_2+a_3\leq a$ such that $am=a_1u_1+a_2u_2+a_3u_4$. Then
\[a\mid(a_1-b_1a_3), (a_2-b_2a_3), -a_3b_3,\]
since $\{u_1, u_2, u_3\}$ is a $\bZ$-basis and the wall relation \eqref{wallrel_eqn2},
and thus
\[
    a_1 + a_2 + a_3 = a_1 + a_2 - a_3 (b_1 + b_2 + b_3) \equiv 0 \pmod{a}.
\]
This implies that $m\in\operatorname{Conv}(u_1,u_2,u_4)\cup\{0\}$, since $0\leq a_1+a_2+a_3\leq a$.

Also, $X_{\Sigma} \to U_{\sigma_0}$ is a  flopping contraction since $K_{X_{\Sigma}} \cdot \cR = 0$ by \eqref{flop_eqn}.
\end{remark}

\begin{exmp}\label{non-toric-exam}
By the elementary number theory, if $b_3\geq (b_1'-1)(b_2'-1)$, then the condition \eqref{b=b+b} will hold. 
In this case, the fiber product $X$ is the toric variety  by Proposition \ref{can3flop_prop}.
%According to the theorem below, in this case, the fiber product $X$ is a toric variety.

On the other hand, taking $(b_1, b_2) = (- 3, - 3 k - 5)$ for $k \in \eZ$, we find that 
the fiber product $X$ is not a toric variety in this case. In fact, it is easily seen that
if $(b_1,b_2)$ does not satisfy \eqref{b=b+b} then $(b_1, b_1 + b_2)$ also does not satisfy \eqref{b=b+b}.

%It is clear that if $(b_1,b_2)$ does not satisfy \eqref{b=b+b}, then $(b_1, b_1 + b_2)$ also does not satisfy \eqref{b=b+b}. Hence we can construct the fiber product $X$, which is not a toric variety in this case; for example, taking $(b_1, b_2) = (- 3, - 3 k - 5)$ for $k \in \eZ$. 
\end{exmp}

%%%%%%%%%%%%%%%%%%%%
%%%%%%%%%%%%%%%%%%%%
%\subsection{Smooth flipping contractions}
%\subsection{Smooth $n$-fold for $n\geq 4$}
%Consider the flip or flop $X_{\Sigma}\rightarrow X_{\Sigma_0}\leftarrow X_{\Sigma'}$ from a smooth simplicial toric $n$-fold. 

%Consider a flipping contraction $X_{\Sigma} \to X_{\Sigma_0}$ from a smooth toric variety, and let $X_{\Sigma'} \to X_{\Sigma_0}$ be the flip of it.

%Assume that $X_{\Sigma_0}$ is affine.

%Since $\{u_k\mid k\neq j\}$ form a $\Z$ basis of $N$ for any $j\in J_+$, we have $b_j\mid b_k$ for any $k\in\{1,\ldots,n+1\}$. In particular, if we assume that $\operatorname{gcd}(b_1,\ldots,b_{n+1})=1$, then $b_j=1$ for any $j\in J_+$.
\subsection{Higher dimensional case}
Since the criterion for being reduced of $X$ only holds for $3$-dimensional case, it seems difficult to say something more about higher dimensional case. 
However, if we assume the smoothness of $X_{\Sigma}$, then by more detailed argument, a useful criterion can also be achieved. 

\begin{thm}\label{smooth is redcued equivalent}
Assume that $X_{\Sigma}$ is smooth of dimension $n$ and $X_{\Sigma_0}$ is affine with the wall relation \eqref{wallrel_eqn2}. Then the fiber product $X=X_{\Sigma}\times_{X_{\Sigma_0}}X_{\Sigma'}$ is the toric variety $X_{\tSigma}$ %is reduced 
if and only if 
\begin{equation}\label{divided}
b_{i}\mid b_{j} \mbox{ or }\ b_{j}\mid b_i
\end{equation}
for any $i$, $j\in J_-$.
\end{thm}

%given $j$, $j'\in J_-$ we have either b_{j}\mid b_{j'} \text{ or }\ b_{j'}\mid b_j

\begin{proof}
The irreducibility of $X$ and the normality of $\Graphc$ are true by Theorem \ref{fpred_thm} and Corollary \ref{gftoric_cor} \eqref{gftoric_cornDsm}. Then by Proposition \ref{fpcri_prop}, the result follows from the {\bf claim} that $\eqref{divided}$ is equivalent to the generalized version of \eqref{Gamma in Gamma'} in Lemma \ref{n=3_reduced_equiv} for any affine piece of $X$.

Recall that $\gcd(b_1, \ldots, b_{n+1}) = 1$. Given $j \in J_{+}$, we have that the set $\{u_k \mid k\neq j\}$ forms a $\bZ$-basis of $N$ since $X_{\Sigma}$ is smooth, so the wall relation \eqref{wallrel_eqn2} implies that $b_j \mid b_k$ for each $1 \leq k \leq n+1$, and thus $b_j=1$ for each $j\in J_+$. %In particular, if we assume that $\operatorname{gcd}(b_1,\ldots,b_{n+1})=1$, then $b_j=1$ for each $j\in J_+$.

For simplicity, we assume that $1\in J_-$ and $n+1\in J_+$. 
%Since the idea of the proof  is similar to the proof of \eqref{n=3_reduced_equiv_1}$\Leftrightarrow$\eqref{n=3_reduced_equiv_2} in Lemma \ref{n=3_reduced_equiv}, so we use the same notation in the proof of Theorem~\ref{fpred_thm}. 
Let $J_+^* =  J_+ \setminus \{n+1\}$ and $J_-^* =  J_- \setminus \{1\}$. We will prove that $b_1\mid b_i$ or $b_i\mid b_1$ for each $i\in J_-^*$ if and only if $U_{\sigma}\times_{U_{\sigma_0}}U_{\sigma'}$ is reduced, where $\sigma = \sigma_{n + 1}$ and $\sigma' = \sigma_1$. %Set $U_{ji} = U_{\sigma_j}\times_{U_{\sigma_0}}U_{\sigma_i}$.

%We will find the equivalent statement of $U_{ji}\coloneqq U_{\sigma_j}\times_{U_{\sigma_0}}U_{\sigma_i}$ being reduced for $i\in J_-$, $j\in J_+$. We will do the case for $|J_+|>2$, and will see the same proof in Lemma \ref{n=3_reduced_equiv} works for $|J_+|=2$. 

%For simplicity, we assume that $1\in J_-$ and $n+1\in J_+$, and 
 
%$u_j^\vee\in\sigma_{n+1}^\vee\setminus\sigma_{1}^\vee$ for $j\in J_+^*$, $-b_1u_i^\vee+b_iu_1^\vee\in\sigma_1^\vee\setminus\sigma_{n+1}^\vee$ for $i\in J_-^*$, and $u_k^\vee\in\sigma_{n+1}^\vee\cap\sigma_1^\vee$ for $k\in J_0$.

%\begin{equation*}
   % u_j^\vee\in\sigma_{n+1}^\vee\setminus\sigma_{1}^\vee,\  -b_1u_i^\vee+b_iu_1^\vee\in\sigma_1^\vee\setminus\sigma_{n+1}^\vee \mbox{ and } u_k^\vee\in\sigma_{n+1}^\vee\cap\sigma_1^\vee
%\end{equation*}
%for $j\in J_+^*$, $i\in J_-^*$ and $k\in J_0$, where $\{u_1^\vee,\cdots,u_n^\vee\}$ is the dual basis of $\{u_1,\cdots,u_n\}$. 

%consider $U_{(n+1)1}$.

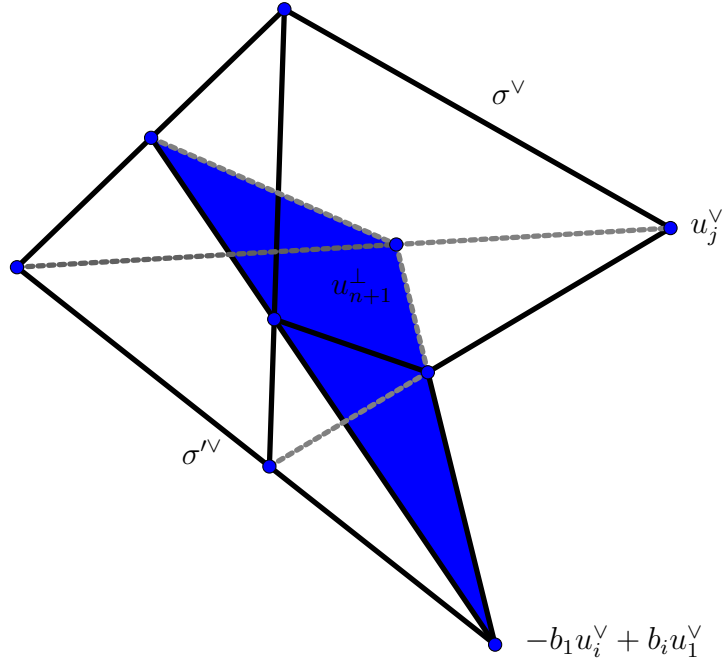
\begin{figure}
\centering
\begin{tikzpicture}[line cap=round,line join=round,>=triangle 45,x=1.0cm,y=1.0cm]
\definecolor{wqwqwq}{rgb}{0.3764705882352941,0.3764705882352941,0.3764705882352941}
\definecolor{yqyqyq}{rgb}{0.5019607843137255,0.5019607843137255,0.5019607843137255}
\definecolor{qqqqff}{rgb}{0.,0.,1.}
\clip(-2.1103707162942564,-6.3) rectangle (18.988760620667172,2.8);
\fill[line width=2.pt,color=qqqqff,fill=qqqqff,fill opacity=0.10000000149011612] (9.161724065488615,-5.939206156058178) -- (6.225292398996466,-1.6095540972294011) -- (8.2674473020913,-2.3159826941244566) -- cycle;
\fill[line width=2.pt,color=qqqqff,fill=qqqqff,fill opacity=0.10000000149011612] (6.225292398996466,-1.6095540972294011) -- (4.588644536840115,0.8036182802596363) -- (7.847786988488664,-0.6157002351336429) -- (8.2674473020913,-2.3159826941244566) -- cycle;
\draw [line width=2.pt] (2.805476719114314,-0.9180609920273426)-- (6.161133230085074,-3.568876409369539);
\draw [line width=2.pt] (2.805476719114314,-0.9180609920273426)-- (6.3603274521974935,2.5142086812943454);
\draw [line width=2.pt] (6.3603274521974935,2.5142086812943454)-- (6.161133230085074,-3.568876409369539);
\draw [line width=2.pt] (6.3603274521974935,2.5142086812943454)-- (11.493409329709845,-0.397091488041015);
\draw [line width=2.pt] (6.161133230085074,-3.568876409369539)-- (9.161724065488615,-5.939206156058178);
\draw [line width=2.pt] (6.225292398996466,-1.6095540972294011)-- (8.2674473020913,-2.3159826941244566);
\draw [line width=2.pt,dash pattern=on 2pt off 2pt,color=yqyqyq] (4.588644536840115,0.8036182802596363)-- (7.847786988488664,-0.6157002351336429);
\draw [line width=2.pt] (8.2674473020913,-2.3159826941244566)-- (9.161724065488615,-5.939206156058178);
\draw [line width=2.pt,dash pattern=on 2pt off 2pt,color=yqyqyq] (7.847786988488664,-0.6157002351336429)-- (8.2674473020913,-2.3159826941244566);
\draw [line width=2.pt] (4.588644536840115,0.8036182802596363)-- (6.225292398996466,-1.6095540972294011);
\draw [line width=2.pt] (6.225292398996466,-1.6095540972294011)-- (9.161724065488615,-5.939206156058178);
\draw [line width=2.pt,dash pattern=on 2pt off 2pt,color=yqyqyq] (6.161133230085074,-3.568876409369539)-- (8.2674473020913,-2.3159826941244566);
\draw [line width=2.pt] (8.2674473020913,-2.3159826941244566)-- (11.493409329709845,-0.397091488041015);
\draw [line width=2.pt,dash pattern=on 2pt off 2pt,color=wqwqwq] (2.805476719114314,-0.9180609920273426)-- (7.847786988488664,-0.6157002351336429);
\draw [line width=2.pt,dash pattern=on 2pt off 2pt,color=yqyqyq] (7.847786988488664,-0.6157002351336429)-- (11.493409329709845,-0.397091488041015);
\draw (11.613400551552333,-0.013816764166816072) node[anchor=north west] {$u_j^\vee$};
\draw (9.41279582930288,-5.548671065582118) node[anchor=north west] {$-b_1u_i^\vee+b_iu_1^\vee$};
\draw (8.97267488485299,1.7066560186827597) node[anchor=north west] {$\sigma^\vee$};
\draw (4.851542405004013,-3.0813263769993933) node[anchor=north west] {$\sigma'^\vee$};
\draw (6.838755154186853,-0.787362666533292) node[anchor=north west] {$u_{n+1}^\perp$};
\begin{scriptsize}
\draw [fill=qqqqff] (2.805476719114314,-0.9180609920273426) circle (2.5pt);
\draw [fill=qqqqff] (6.161133230085074,-3.568876409369539) circle (2.5pt);
\draw [fill=qqqqff] (6.3603274521974935,2.5142086812943454) circle (2.5pt);
\draw [fill=qqqqff] (11.493409329709845,-0.397091488041015) circle (2.5pt);
\draw [fill=qqqqff] (9.161724065488615,-5.939206156058178) circle (2.5pt);
\draw [fill=qqqqff] (4.588644536840115,0.8036182802596363) circle (2.5pt);
\draw [fill=qqqqff] (7.847786988488664,-0.6157002351336429) circle (2.5pt);
\draw [fill=qqqqff] (6.225292398996466,-1.6095540972294011) circle (2.5pt);
\draw [fill=qqqqff] (8.2674473020913,-2.3159826941244566) circle (2.5pt);
\end{scriptsize}
\end{tikzpicture}
\caption{$\sigma^\vee$ and $(\sigma')^\vee$ which are cut by a hyperplane in $M$}
\end{figure}
%Then the dual cone cut by a hyperplane in $M$ as illustrated in the following figure

Assume that $U_{\sigma}\times_{U_{\sigma_0}}U_{\sigma'}$ is reduced. Observe that $u_j^\vee\in\sigma^\vee\setminus\sigma'^\vee$ spans a ray of $\sigma^\vee$ for each $j\in J_+^*$, where $\{u_1^\vee,\ldots,u_n^\vee\}$ is the dual basis of $\{u_1,\ldots,u_n\}$.  Notice that the generating set $\sA$ can be selected to be $\sA_0\cup\{u_j^\vee\mid j\in J_+^*\}$, since $\sigma$ is smooth. Set
\[\Gamma=M\cap\overline{\sigma'^\vee\setminus\sigma^\vee}\quad \mbox{and} \quad \Gamma'=\Gamma\cap\Cone(u_{n+1})^\perp.\]
Then $\gen{z,u_{n+1}}>0$ for each $z\in\Gamma\setminus\Gamma'$. 
%For any
%\[z\in\Gamma\coloneqq M\cap(\overline{\sigma'^\vee\setminus\sigma^\vee}\setminus u_{n+1}^\perp),\]
%$\gen{z,u_{n+1}}>0$. 
Since $\gen{u_j^\vee,u_{n+1}}=-1$, we still have $\gen{z+u_j^\vee,u_{n+1}}\geq 0$ and thus $z+u_j^\vee\in\Gamma$ for all $j\in J_+^*$. By the same argument in Lemma \ref{n=3_reduced_equiv}, if $U_{\sigma}\times_{U_{\sigma_0}}U_{\sigma'}$ is reduced, then 
\begin{equation}\label{smooth eq1}
\Gamma\subseteq \sfS_{\sigma_0}+\Gamma'.
\end{equation} 
Let $\Gamma''=\Gamma'\cap \bigcap_{j\in J_+^*} \Cone(u_j)^\perp$. For any $z\in\Gamma'\setminus\Gamma''$, there is a $j_1\in J_+^*$ such that $\gen{z,u_{j_1}}>0$. If $|J_+|\neq 2$, then we can choose another index $j'\in J_+^*\setminus\{j_1\}$, which implies that
\[(z-u_{j_1}^\vee+u_{j'}^\vee)+u_{j_1}^\vee=z+u_{j'}^\vee\]
and $z-u_{j_1}^\vee+u_n^\vee\in\Gamma'$. Again, by the same argument, we have 
\begin{equation}\label{smooth eq2}
\Gamma'\setminus\Gamma''\subseteq\left(\sfS_{\sigma_0}\setminus 0\right)+\Gamma''
\end{equation}
if $U_{\sigma}\times_{U_{\sigma_0}}U_{\sigma'}$ is reduced. Combining \eqref{smooth eq1} and \eqref{smooth eq2}, we get
\begin{equation}\label{smooth eq3}
\Gamma\subseteq \sfS_{\sigma_0}+\Gamma''.
\end{equation}

Conversely, we claim that \eqref{smooth eq3} (resp.~\eqref{smooth eq1}) implies that $U_{\sigma}\times_{U_{\sigma_0}}U_{\sigma'}$ is reduced when $|J_+|\neq 2$ (resp. $|J_+|=2$). %Following the notation in the proof of Proposition \ref{gfbir_prop}, given
As in the proof of Theorem \ref{gfbir_prop}, we pick 
\begin{equation}
    (\alpha_0 + \alpha + \alpha') - (\beta_0 + \beta + \beta') \in L_\sB
\end{equation}
where $\alpha_0, \beta_0 \in \eZ^{\sA_0}$, $\alpha, \beta \in \eZ^{\sA \setminus \sA_0}$ and $\alpha', \beta' \in \eZ^{\sA' \setminus \sA_0}$. If $|J_+|\neq 2$, by the decomposition \eqref{smooth eq3}, then we may assume that $\alpha'$, $\beta'\in\Gamma''$. Since $\sfS_{\sigma}$ is generated by $\{u_i^\vee\mid 1\leq i\leq n\}$ as a semigroup, we may assume that $\alpha_0$, $\beta_0\in\operatorname{Cone}(u_i^\vee\mid i\in J_-\cup J_0)$. In this case, since 
\begin{equation*}
    \alpha_0, \beta_0, \alpha', \beta'\in\bigcap_{j\in J_+^*}u_j^\perp \quad\mbox{and}\quad \alpha, \beta\in\bigcap_{i\in(J_-\cup J_0)}u_i^\perp,
\end{equation*}
%$\alpha_0$, $\beta_0$, $\alpha'$, $\beta'\in\bigcap_{j\in J_+^*}u_j^\perp$ and $\alpha$, $\beta\in\bigcap_{i\in(J_-\cup J_0)}u_i^\perp$, 
we conclude that $\alpha=\beta$ and thus
\[x^{\alpha_0+\alpha+\alpha'}-x^{\beta_0+\beta+\beta'}=x^\alpha(x^{\alpha_0+\alpha'}-x^{\beta_0+\beta'})\in I_\sA + I_{\sA'}.\]
If $|J_+|=2$, then $\alpha$, $\beta\in 
\Z_{\geq 0} \Cone(u_\ell)^\perp$ where $J_+=\{\ell,n+1\}$. Then the same argument in Lemma~\ref{n=3_reduced_equiv} will work.

We can translate the condition~\eqref{smooth eq3} (resp.~\eqref{smooth eq1}) to the following condition:
\begin{enumerate}
    \item[(\hypertarget{smooth}{$\spadesuit$})] Given $z=\sum_{i=1}^n z_iu_i^\vee\in\Gamma$, that is, $z_1\leq 0$, $z_i\geq 0$ for $i\neq 1$ and 
    \begin{equation*}
    \sum_{i=1}^n b_iz_i\leq 0,
    \end{equation*}
    there exists $z'=\sum_{i\in (J_-\cup J_0)} z_i'u_i^\vee\in\Gamma''$ (resp.~$z'=\sum_{i=1}^n z_i'u_i^\vee\in\Gamma'$) such that $z-z'\in\sfS_{\sigma_0}$, that is, there exists $z_1'\leq 0$, $z_i'\geq 0$ for all $i\neq 1$ such that
    \begin{equation*}
    \sum_{i} b_iz_i'=0,
    \end{equation*}
    and $z_i\geq z_i'$ for all $i$.
\end{enumerate}
We claim that if (\hyperlink{smooth}{$\spadesuit$}) holds, then $b_1\mid b_i$ or $b_i\mid b_1$ for each $i\in J_-^*$. 
%for each $i\in J_-^*$, $b_1\mid b_i$ or $b_i\mid b_1$ if (\hyperlink{smooth}{$\spadesuit$}) holds. 
Indeed, let $g=\operatorname{gcd}(b_1,b_i)>0$ and $b_1=-gb_1'$, $b_i=-gb_i'$. There exists integers $0\leq n_1<b_i'$ and $n_i$ such that $b_i'n_i=b_1'n_1+1$. Note that 
\begin{equation}\label{divide_ineq}
    n_i=\dfrac{b_1'n_1+1}{b_i'}\leq\dfrac{b_1'(b_i'-1)+1}{b_i'}=b_1'+\dfrac{1-b_1'}{b_i'}\leq b_1'
\end{equation}
and the equality holds only when $b_1'=1$ and $n_1=b_i'-1$. Consider $z=-n_1u_1^\vee+n_iu_i^\vee\in\overline{\Gamma}$. Then there exists $z_i'\geq 0\geq z_1'$ such that $-n_1\geq z_1'$ and $n_i\geq z_i'$ by (\hyperlink{smooth}{$\spadesuit$}) and $b_1'z_1'+b_i'z_i'=0$. If $b_i'\neq 1$, then $n_1>0$. Since $b_i'\mid z_1'$ and $z_1'\leq -n_1<0$, we must have $z_1'\leq -b_i'$, and thus
\[b_1'\geq n_i\geq z_i'=\dfrac{-b_1z_1'}{b_i'}\geq b_1'.\]
We conclude that $n_i=b_1'$, and thus $b_1'=1$, since the equality in \eqref{divide_ineq} holds.

%By the above claim and symmetry, we have \eqref{divided} holds if $X$ is reduced.

Conversely, we will show that \eqref{divided} implies (\hyperlink{smooth}{$\spadesuit$}). Given $z\in\Gamma$, we define $\widetilde{z}=\sum_{i\in J_-}\widetilde{z}_iu_i^\vee$ by setting
\[\widetilde{z}_1=\lfloor\frac{-1}{b_1}\sum_{i\in J_-^*}b_iz_i\rfloor \quad\mbox{and}\quad \widetilde{z}_i=z_i\quad \mbox{for } i\in J_-^*.\]
Then $\sum_{i\in J_-}b_i\widetilde{z}_i\leq 0$ by construction, so that $\widetilde{z}\in\Gamma$. %Then $\widetilde{z}\in\Gamma$, since $\sum_{i\in J_-}b_i\widetilde{z}_i\leq 0$ by construction.
If there exists $z'=\sum_{i\in J_-}z_i'u_i^\vee\in\Gamma''\subseteq\Gamma'$ such that $\widetilde{z}-z'\in\sfS_{\sigma_0}$, then
\[z-z'=(z-\widetilde{z})+(\widetilde{z}-z')\in\eZ u_1^\vee+\sfS_{\sigma_0}\subseteq\sfS_{\sigma_0}.\]
So it suffices to show that
\begin{enumerate}
    \item[(\hypertarget{smooth'}{$\spadesuit'$})] Given $z=\sum_{j\in J_-}^n z_1u_i^\vee$, where $z_i\geq 0$ for $i\in J_-^*$ and $z_1=\lfloor\frac{-1}{b_1}\sum_{j\in J_-}b_jz_j\rfloor$, there exists $z'=\sum_{i\in J_-} z_i'u_i^\vee\in\Gamma''$ such that $z-z'\in\sfS_{\sigma_0}$.
\end{enumerate}
Observe that condition (\hyperlink{smooth'}{$\spadesuit'$}) holds when $b_i=-1$ for some $i\in J_-$. Indeed, we can take 
\[z_i'=\sum_{k\in J_-\setminus\{i\}}b_kz_k\leq z_i \quad\mbox{and}\quad z_k'=z_k\quad\text{for } k\neq J_-\setminus\{i\}.\]
Moreover, since we can divide the greatest common divisor of $\{b_i\}_{i\in J_-}$ in the condition (\hyperlink{smooth}{$\spadesuit$}) and go back to the case when $b_i=-1$ for some $i\in J_-$, we have \eqref{divided} implies (\hyperlink{smooth'}{$\spadesuit'$}), and thus implies (\hyperlink{smooth}{$\spadesuit$}). Hence $X$ is reduced when \eqref{divided} holds.
\end{proof}

%%%%%%%%%%%%%%%%%%%%
%%%%%%%%%%%%%%%%%%%%
\appendix \section{Remarks on $K$-equivalent toric varieties}\label{app_sec}

Let $X=X_\Sigma$ and $X'=X_{\Sigma'}$ be two simplicial toric
varieties with at most terminal singularities such that $X$ and $X'$ are $K$-equivalent, denoted by $X=_K
X'$. We know that this is equivalent to that $\hbox{shed}\,\Sigma
= \hbox{shed}\,\Sigma'$, or in other words, the fans $\Sigma$ and
$\Sigma'$ give rise to different triangulations of the same
polyhedron.

Let $f\colon X\dashrightarrow X'$ be a toric flop given by
\[\begin{tikzcd}
X=X_\Sigma \arrow[rd, "\phi_\cR"'] && X'=X_{\Sigma'} \arrow[ld, "\phi'"] \\
& \bar X=X_{\Sigma_0} &
\end{tikzcd}\]
and $\tau$ be a wall with $V(\tau) \in \cR$ whose wall relation is $\sum_{i = 1}^{n + 1} b_i u_i = 0$. We may assume that $J_- = \{1,\ldots,\alpha\}$, $J_+ = \{\beta+1,\ldots,n+1\}$ and $b_{n+1}=1$.
Note that all the primitive vectors should all lie in an  affine hyperplane
of $N_{\mathbb{Q}}$ for toric flops. The exceptional set $Z$ of $\phi_\cR$ corresponding to $\Cone(u_1,\ldots,u_\alpha)$ under $\phi_\cR$ of dimension $n-\alpha$ is mapped to
$S\coloneqq\phi_\cR(Z)\subset \bar{X}$ corresponding to 
$\Cone(u_1,\ldots,u_\alpha,u_{\beta+1},\ldots,u_{n+1})$ of dimension $\beta-\alpha$. The map $Z\to S$ is a
bundle with fiber covered by a weighted projective space
$\tilde{\mathbb{P}}^{n-\beta}$ through a finite morphism (cf.
\cite{Matsuki}, 14-2-3). Similar statements hold for $\phi' \colon X'\to \bar X$ with the exceptional set $Z'$ of $\phi'$ fibered over
$S$ with fiber covered by
$\tilde{\mathbb{P}}^{\alpha-1}$. 

%Here we prefer to ``go above'' instead of to ``go down''. That is,
%consider the unique vector $u_0\in N=\mathbb{Z}^n$:
%$$
%u_0=\sum\nolimits_{i=1}^\alpha (-b_i)u_i =
%\sum\nolimits_{i=\beta+1}^{n} b_i u_i + u_{n+1}.
%$$
%The barycentric subdivision of either decomposition of $\Sigma$
%leads to the same cone decomposition and defines a toric variety
%$Y$ and weighted blowing-ups 
%\[\begin{tikzcd}
%& Y \arrow[ld,"\psi"'] \arrow[rd,"\psi'"] &\\
%X && X'\\
%\end{tikzcd}\]
%The exceptional divisor of $\psi$, $\psi'$ and
%$\bar\psi\coloneqq\phi\circ\psi= \phi'\circ\psi':Y\to \bar{X}$ is the
%same one $E$, which is fibered over $S$ with fiber a
%weighted projective space
%$\tilde{\mathbb{P}}^{n-\beta}\times\tilde{\mathbb{P}}^{\alpha-1}$.
%The map $\psi$ is the blowing-down of the first factor and $\psi'$ is
%the second. 

\begin{proposition}\label{smooth=ordinary}
Any smooth toric flop is ordinary. 
\end{proposition}
\begin{proof}
The smoothness condition tells
us that the primitive generators $u_1,\ldots, u_n$ form a
$\mathbb{Z}$-basis of the lattice $N$ and so do
$u_1,\ldots,u_{n-1},u_{n+1}$. When we represent $u_{n+1}$ as a
$\mathbb{Z}$-linear combination of $u_1,\ldots, u_n$ and $u_n$ as
a $\mathbb{Z}$-linear combination of $u_1,\ldots,u_{n-1},u_{n+1}$
simultaneously, we can get that $b_i = -1$ for $i =
1,\ldots,\alpha$ and $b_i = 1$ for $i = \beta + 1,\ldots, n$. 

All $u_1,\ldots,u_n,u_{n+1}$ should all lie in
an affine hyperplane of $N_{\mathbb{Q}}$. This implies that
$$
- \sum\nolimits_{i=1}^\alpha b_i = \sum\nolimits_{i =
\beta+1}^nb_i + 1
$$
and thus $\alpha = n + 1 - \beta$.

Translating these data to the Reid's diagram, we have that $Z\to
S$ is a bundle with fiber a projective space
$\mathbb{P}^{n-\beta}$ and $Z'\to S$ is a bundle with fiber a
projective space $\mathbb{P}^{\alpha-1}$. Note that $n - \beta =
\alpha - 1$. It illustrates that the whole diagram for this case
is an ordinary $\mathbb{P}^{\alpha-1}$-flop. Hence we complete the
proof.
\end{proof}

By extending Reid's argument, we may decompose $K$-equivalent birational maps into toric
flops.

\begin{theorem}\label{K-equiv}
Let $X=X_\Sigma$ and $X'=X_{\Sigma'}$ be two simplicial toric
varieties with at most terminal singularities such that $X=_K X'$.
Then the birational map $f\colon X\dashrightarrow X'$ can be factorized
into toric flops.
\end{theorem}
\begin{proof}
We know that $X=_K X'$ is equivalent to that the fans $\Sigma$ and
$\Sigma'$ give rise to different simplicial triangulations of the
same polyhedron fan $\bar\Sigma$. Also, the condition of having
terminal singularities shows that the possible lattice points in
$\hbox{shed}\,\Sigma$ and $\hbox{shed}\,\Sigma'$ are $0$ and
primitive generators of $\Sigma$ and $\Sigma'$. From the above
observation, we conclude that $\Sigma$ and $\Sigma'$ have the same
edges and thus $X$ and $X'$ are isomorphic in codimension one.

Let $\bar X = X_{\bar\Sigma}$ and $\phi, \phi'$ be the toric morphisms with respect to the simplicial triangulations of $\bar\Sigma$:
\[\begin{tikzcd}
X \arrow[rd, "\phi"'] && X' \arrow[ld, "\phantom{}\phi'"] \\
& \bar X &
\end{tikzcd}\]
Also, we see easily that both $\phi$ and $\phi'$ are
crepant morphisms.

We take an ample divisor $H'$ in $X'$ and consider $H$ as the
proper transform of $H'$ in $X$. We run the $(K_X + H)$-minimal
model program of the morphism $\phi$. Since $K_X$ is
$\phi$-trivial and $H$ is not $\phi$-nef, we see that there must
exist $(K_X + H)$-extremal rays $\cR = \mathbb{R}^+[C]$ with
$\phi(C) = {\rm pt}$ and $H.C < 0$. This implies that the
corresponding extremal contraction $\phi_\cR$ is small. For if it is
divisorial or fiber type, then we may represent $\cR$ by an
irreducible curve $C$ which does not lie in the $f$ exceptional
loci $Z \subseteq X$. But since $H$ is ample outside $Z$, we may
make $H$ and $C$ transversal hence $H.C \ge 0$, a contradiction.

Now we perform the toric flip $f_\cR\colon X\dashrightarrow X^+$ of $\phi_\cR$. Since $K_X.C = 0$, $f_\cR$ is indeed a toric flop
and hence $f^+ \colon X^+\to\bar X$ still inherits all the relevant
properties of $X$. In particular we still have $X^+ =_K X'$. By
the termination of toric flips we conclude the proof of the theorem.
\end{proof}

The termination of toric flips  was stated with
sketched proof in \cite{Matsuki}. For the reader's convenience, we
provide the details here. 

\begin{theorem}
Let $X=X_\Sigma$ be a complete simplicial toric variety and $D =\sum_{\rho \in \Sigma(1)} a_\rho D_\rho$ be a $\mathbb{Q}$-divisor with $ 0
\leq a_\rho < 1$. Then we have the termination of $(K_X + D)$-flips.
\end{theorem}

\begin{proof}
Recall that $\hbox{shed}\,\Sigma = \{u \in \Sigma\,|\,
\phi_{K_X}(u) \leq 1\}$ where $\phi_{K_X}$ is the corresponding
piecewise linear function of $K_X$. Now we bring up a more general
notion of relative shed of $\Sigma$ with respect to $D$ which is
defined by the set $\{u \in \Sigma\,|\, \phi_{K_X + D}(u) \leq
1\}$ and is denoted by $\hbox{shed}_D\,\Sigma$. Our strategy is to
give a similar geometric criterion about $\hbox{shed}_D\,\Sigma$
for the condition of $(K_X + D).V(\tau) < 0$ with $\tau$ a wall in $\Sigma$.

Indeed, when we restrict ourselves on an affine cone $\sigma =
\Cone (u_1,\ldots,u_n)$ in $\Sigma$ where $u_i$ is the primitive vector of the 1-face $\rho_i$,
$\hbox{shed}_D\,\Sigma|_{\sigma}$ is equal to the convex hull
$$\operatorname{Conv}\left(0, \frac{u_1}{1-a_{\rho_1}},\ldots,
\frac{u_n}{1-a_{\rho_n}}\right).
$$
%$$
%\left[0, \frac{u_1}{1-a_{e_1}},\ldots,
%\frac{u_n}{1-a_{e_n}}\right].
%$$

By the same argument in \cite{Reid83}, we get that if $(K_X + D).V(\tau)
< 0$ then $\hbox{shed}_D\,\Sigma$ has a bridge along $\tau$ and if
$(K_X + D).V(\tau) > 0$ then $\hbox{shed}_D\,\Sigma$ has a gutter
along $\tau$. Hence if $X_{\Sigma'}$ is the $(K_X + D)$-flipped
variety, then
$$
\hbox{Vol}(\hbox{shed}_D\,\Sigma) >
\hbox{Vol}(\hbox{shed}_D\,\Sigma').
$$

However the values of both volumes are in the discrete set
$$
(\prod\nolimits_{i=1}^n (1 - a_{\rho_i}))^{-1}(n!)^{-1}\mathbb{N}.
$$
This implies that it is impossible to have an infinite sequence of
$(K_X + D)$-flips.
\end{proof}

\bibliographystyle{alpha}

\end{document}